\documentclass[a4paper,11pt]{amsart}
\usepackage{amsmath}
\usepackage{amsthm,amsfonts,amssymb,graphics}
\usepackage[foot]{amsaddr}
\usepackage{verbatim}
\usepackage{hyperref}
\usepackage{xcolor}
\usepackage[normalem]{ulem}
\usepackage{enumerate}
\usepackage{tikz}
\usepackage{geometry}
\usepackage{dsfont}
\usepackage{centernot}
\usepackage{bbm}

\newtheorem{theorem}{Theorem}[section]
\newtheorem{lemma}[theorem]{Lemma}
\newtheorem{proposition}[theorem]{Proposition}
\newtheorem{corollary}[theorem]{Corollary}
\newtheorem{assumption}[theorem]{Assumption}

\theoremstyle{remark}
\newtheorem{remark}[theorem]{Remark}

\numberwithin{equation}{section}

\newcommand \id{\mathds 1}
\newcommand {\R} {\mathbb{R}}
\newcommand {\D} {\mathcal{D}}
\newcommand {\E} {\mathbb{E}}
\newcommand {\N} {\mathbb{N}}
\newcommand {\Z} {\mathbb{Z}}
\renewcommand{\P} {\mathbb{P}}
\newcommand {\U} {\mathcal{U}}
\newcommand {\Var} {\mathrm{Var}}

\newcommand {\Cov} {\mathrm{Cov}}
\newcommand {\DC} {\mathrm{DC}}
\newcommand {\dist} {\mathrm{dist}}

\newcommand {\eps}{\varepsilon}

\newcommand{\ind}{\mathds{1}}

\begin{document}
\title[A CLT for the number of excursion set components of Gaussian fields]{A central limit theorem for the number of excursion set components of Gaussian fields} 
\author{Dmitry Beliaev\textsuperscript{1}}
\address{\textsuperscript{1}Mathematical Institute, University of Oxford}
\email{belyaev@maths.ox.ac.uk}
\author{Michael McAuley\textsuperscript{2}}
\address{\textsuperscript{2}Department of Mathematics and Statistics, University of Helsinki}
\email{michael.mcauley@helsinki.fi}
\author{Stephen Muirhead\textsuperscript{3}}
\address{\textsuperscript{3}School of Mathematics and Statistics, University of Melbourne}
\email{smui@unimelb.edu.au}
\begin{abstract}
For a smooth stationary Gaussian field on $\R^d$ and level $\ell \in \R$, we consider the number of connected components of the excursion set $\{f \ge \ell\}$ (or level set $\{f = \ell\}$) contained in large domains. The mean of this quantity is known to scale like the volume of the domain under general assumptions on the field. We prove that, assuming sufficient decay of correlations (e.g.\ the Bargmann-Fock field), a central limit theorem holds with volume-order scaling. Previously such a result had only been established for `additive' geometric functionals of the excursion/level sets (e.g.\ the volume or Euler characteristic) using Hermite expansions. Our approach, based on a martingale analysis, is more robust and can be generalised to a wider class of topological functionals. A major ingredient in the proof is a third moment bound on critical points, which is of independent interest.
\end{abstract}
\date{\today}
\thanks{}

\maketitle

\section{Introduction}\label{s:introduction}
Let $f:\R^d\to\R$ be a smooth stationary centred Gaussian field. We consider the geometry of the (upper)-excursion sets and level sets, defined respectively as
\begin{displaymath}
\{f\geq\ell\}:=\big\{x\in\R^d:f(x)\geq\ell\big\}\quad\text{and}\quad\{f=\ell\}:=\big\{x\in\R^d:f(x)=\ell\big\} , \quad \ell \in \R.
\end{displaymath}
In particular, we are interested in the number of connected components of these sets contained in large domains (the `component count').

The geometry of Gaussian excursion/level sets is a well-studied topic, with applications to a range of subjects including cosmology \cite{bbks86,pra19}, medical imaging \cite{wor96}, and quantum chaos~\cite{js17}. The component count is of high importance to these applications; to give an example, physical theories predict that the Cosmic Microwave Background Radiation can be modelled as a realisation of an isotropic Gaussian field on the sphere, and in \cite{pra19} this prediction was tested statistically using numerical simulations of the component count of a Gaussian field.

A second motivation to study the component count comes from recent progress in understanding the connectivity of Gaussian field excursion sets (see, e.g., \cite{bg17, rv20, mv20, mrv23} for a selection of recent results, and more generally \cite{gri99} for background on classical percolation theory). The component count is significant for this study: in classical percolation, smoothness properties of the mean number of open clusters (which corresponds to the excursion set count in the Gaussian setting) are related to the uniqueness of infinite connected components~\cite{akn87}, and broadly analogous results have recently been proven for Gaussian fields~\cite{bmm20a}.

In dimension $d=1$ the component count reduces to the number of level crossings, which is a classical topic in probability theory \cite{kra06}.

\subsection{Existing results on the component count}
Recall that $f$ is a smooth stationary centred Gaussian field on $\R^d$, and let $K(x) = \E[f(0)f(x)]$ be its covariance kernel. For $R>0$ and $\ell\in\R$ we denote by $N_\mathrm{ES}(R,\ell)$ and $N_\mathrm{LS}(R,\ell)$, respectively, the number of connected components of $\{f\geq\ell\}$ and $\{f = \ell\}$ which are contained in the cube $\Lambda_R = [-R,R]^d$ (i.e.\ which intersect this set but not its boundary). The precise choice of domain $\Lambda_R$, and the choice to exclude boundary components, are mainly for concreteness, and could be modified with minimal change to the results or proof. For simplicity, we shall often write $N_\star(R,\ell)$, $\star \in \{\mathrm{ES},\mathrm{LS}\}$, to refer collectively to these component counts.

We are interested in the asymptotics of $N_\star(R, \ell)$ as $R \to \infty$. The first-order convergence (i.e.\ law of large numbers) was established by Nazarov and Sodin \cite{ns16}: under very mild conditions on the field, as $R \to \infty$,
\begin{equation}
\label{e:lln}
\frac{N_\star(R,\ell)}{\mathrm{Vol}(\Lambda_R)}\to \mu \quad \text{in $L^1$ and almost surely,}
\end{equation}
for a constant $\mu = \mu_\star(\ell)>0$ which depends on the law of $f$. Although the result in \cite{ns16} was stated only for the nodal set (i.e.\ the level set at $\ell=0$), the proof goes through verbatim for excursion/level sets at all levels (see, e.g., \cite{bmm20}). The law of large numbers has since been extended to other quantities related to the component count \cite{kw18,bw18,sw19,wig21}.

A natural next step is to investigate the second-order properties of $N_\star(R, \ell)$, which are expected to depend strongly on the covariance structure of the field. Here we focus on the \textit{short-range correlated} case in which $K \in L^1(\R^d)$; an important example is the \textit{Bargmann-Fock field} with $K(x) = e^{-|x|^2 / 2}$ (see \cite{bg17} for background and motivation). In this case it is expected that $N_\star(R,\ell)$ satisfies a central limit theorem (CLT) with volume-order scaling $\asymp R^d$. This has previously been established for various `additive' geometric functionals of the excursion/level sets (for instance their volume or Euler characteristic \cite{kl01,el16,an17,mul17,kv18}), and it is also known in the case of an i.i.d.\ Gaussian field on $\Z^d$ \cite{cg84,zha01,pen01} (where the component count is equivalent to the number of clusters in classical site percolation).

Thus far, progress on understanding second-order properties has been limited to bounds on the variance which are mostly sub-optimal and apply only to planar fields.

In \cite{ns20} Nazarov and Sodin proved a polynomial lower bound
\[ \Var[N_\star(R,\ell)] \geq c R^\eta \]
valid for general planar fields with polynomially decaying correlations (to be more precise, they considered families $(f_n)_{n\geq 1}$ of Gaussian fields defined on the sphere $\mathbb{S}^2$ which converge locally, and only considered the nodal set, but we expect the proof extends, up to boundary effects, to the Euclidean setting we consider here). The exponent $\eta > 0$ was not quantified but is small. 

In \cite{bmm22} sharper results were proven for planar fields under stronger conditions. More precisely, if $f$ is short-range correlated, and if $\int K(x) dx \neq 0$ and $\frac{d}{d\ell}\mu_\star(\ell)\neq 0$ (recall that $\mu_\star(\ell)$ is the limiting constant in \eqref{e:lln}), then
\[ \Var[N_\star(R,\ell)]\geq cR^2 . \]
Further, in \cite{bmm20a} the condition $\frac{d}{d\ell}\mu_\star(\ell)\neq 0$ was shown to hold for a large range of levels (including the zero level when $\star=\mathrm{ES}$).

Turning to upper bounds, it is straightforward to establish that (in all dimensions)
\begin{equation}
	\label{e:ub}
	\Var[N_\star(R,\ell)] \le c R^{2d} 
\end{equation}
 using a comparison with critical points. More precisely, since each excursion (resp.\ level) set component contains (resp.\ surrounds) at least one critical point, the component count in a compact domain is bounded by the number of critical points in the domain. Since the latter quantity has a second moment of order $\asymp R^{2d}$ \cite{eli85,ef16}, we deduce \eqref{e:ub}. Note that this bound is only expected to be attained for very degenerate Gaussian fields (see \cite{bmm20,bmm22} for examples).
 
 Various concentration bounds have also been established for $N_\star(R,\ell)$ \cite{rv19, bmr20, pri20}, but these do not lead to improved bounds on the variance in the short-range correlated case. Related questions have also been studied in the `sparse' regime $\ell_R \to \infty$ as $R \to \infty$ \cite{tk18}.

\subsection{CLT for the component count}\label{s:Main results}
Our main result establishes a CLT for $N_\star(R,\ell)$ with volume-order scaling, assuming sufficient decay of correlations (e.g.\ the Bargmann-Fock field).

We assume that $f$ has a \textit{spatial moving average representation}
\begin{equation}
    \label{e:wnr}
f = q \ast W ,
\end{equation}
where $q \in L^2(\R^d)$ is Hermitian (i.e.\ $q(x) = q(-x)$), $W$ is the white noise on $\R^d$, and $\ast$ denotes convolution. This representation always exists in the short-range correlated case $K \in L^1(\R^d)$, since one can choose $q = \mathcal{F}^{-1}[\sqrt{\mathcal{F}[K]}]$, where $\mathcal{F}[\cdot]$ denotes the Fourier transform. The covariance kernel of $f$ is $K = q \ast q$.

We impose the following assumptions on the kernel $q$: 
\begin{assumption}
\label{a:clt}
$\,$
\begin{enumerate}
    \item (Non-degeneracy) $q\ne 0$.
    \item (Smoothness) $\partial^\alpha q \in L^1(\R^d)\cap C(\R^d)$ for every $|\alpha| \le 4$ and $K=q\ast q\in C^{10}$.
    \item (Decay) There exist $\beta > 9d$ and $c \ge 1$ such that, for all $|x| \ge 1$,
\[  \max_{|\alpha| \le 2} |\partial^\alpha q(x)|  \le c |x|^{-\beta} . \]
\end{enumerate}
\end{assumption}

Since we assume $q \in L^1(\R^d)$, for positive $q$ the decay rates of $q$ and $K = q \ast q$ are comparable up to constants. In general, $K$ decays at least as quickly as $q$, but due to cancellations it could decay more quickly. Hence, roughly speaking, Assumption \ref{a:clt} demands that correlations decay polynomially with exponent $\beta > 9d$. In particular, the Bargmann-Fock field satisfies Assumption \ref{a:clt}. See Remark~\ref{r:beta} for an explanation of how the condition $\beta > 9d $ arises, and how it can be weakened.

Assumption \ref{a:clt} implies that $f$ is $C^4$-smooth almost surely (this follows from Kolmogorov's theorem, see \cite[Appendix~A]{ns16}). This degree of smoothness may seem strong in comparison to other works, e.g.\ \cite{ns20,bmm22}. However, a key novelty of our approach is that we exploit a \textit{third} moment bound on critical points (see Theorem \ref{t:tmb}), whose proof requires fourth-order smoothness.

Most results for level sets of random fields require some non-degeneracy of the field; it is usually required that the joint distribution of the field and some of its derivatives at a point is non-degenerate. In particular, this is needed for the law of large numbers in \eqref{e:lln}. Under our assumptions the spectral measure has an open set in its support, which is a sufficient condition for non-degeneracy (see Appendix~\ref{app:Gaussian}).

Our main result is the following:
\begin{theorem}[CLT for the component count]
\label{t:clt}
Suppose Assumption \ref{a:clt} holds. Let $\ell \in \R$ and $\star \in \{\mathrm{ES},\mathrm{LS}\}$. Then there exists $\sigma = \sigma_\star(\ell) \ge 0$ such that, as $R \to \infty$,
\begin{displaymath}
\frac{\Var[N_\star(R,\ell)]}{\mathrm{Vol}(\Lambda_R)}\to\sigma^2\qquad\text{and}\qquad \frac{N_\star(R, \ell) -  \E[N_\star(R,\ell)]}{\sqrt{\mathrm{Vol}(\Lambda_R)}}\xrightarrow{d} \sigma Z
\end{displaymath}
where $Z$ is a standard normal random variable.
\end{theorem}

We prove Theorem \ref{t:clt} by generalising an argument due to Penrose \cite{pen01}; the rough idea is to obtain a martingale representation for the component count by resampling portions of the white noise appearing in \eqref{e:wnr}, and apply a martingale CLT. Penrose developed this argument to study the number of clusters in classical percolation (among other applications). In our setting there are additional technical obstacles to overcome, since (i) resampling the white noise in a given region affects the values of the field, and a priori also the topology of the components, at arbitrarily large distances, and (ii) the component count has delicate stability properties in continuous space.

It is interesting to compare this approach to the strategy used to prove all previously known CLTs for `additive' geometric functionals of Gaussian field excursion sets (such as the volume or Euler characteristic) \cite{kl01,el16,an17,mul17,kv18}, which relied on expansions over Hermite polynomials. It appears very challenging to extend the `Hermite expansion' method to non-additive functionals such as the component count; by contrast, we believe that our approach extends naturally to a wider class of topological functionals, additive or otherwise. For instance, one could consider the number of components with a given diffeomorphism type~\cite{sw19}.

On the other hand, the Hermite expansion approach is `robust' in a different sense: it can be successfully applied to some strongly correlated fields which are excluded from our results. It remains a significant ongoing challenge to understand non-additive geometric functionals of strongly correlated fields.

\subsection{Positivity of the limiting variance}
The CLT stated in Theorem \ref{t:clt} does not guarantee that the limiting variance $\sigma^2 = \sigma^2_\star(\ell)$ is strictly positive, and if $\sigma = 0$ the result implies only that $\Var[N_\star(R,\ell)] = o(R^d)$.

Our next result confirms that $\sigma > 0$ under a mild additional condition on the field:

\begin{theorem}
\label{t:posvar}
Suppose that Assumption \ref{a:clt} holds, and in addition that $\int q(x) dx > 0$. Let $\star \in \{\mathrm{ES},\mathrm{LS}\}$ and $\ell \in \R$. Then 
\[ \sigma_\star(\ell) > 0 \]
where $\sigma_\star(\ell)$ is the constant from Theorem \ref{t:clt}. 
\end{theorem}

In particular, for the Bargmann-Fock field this result confirms that $\Var[N_\star(R,\ell)]$ is of volume order for \textit{all} levels, whereas previously this was only known (in the planar case) for levels such that $\frac{d}{d\ell} \mu_\star(\ell) \neq 0$ \cite{bmm22}, which is necessarily violated at (at least) one level.

To prove Theorem \ref{t:posvar} we exploit a (semi)-explicit representation for $\sigma$ (see \eqref{e:sigma}) that is a by-product of the proof of the CLT. While the `Hermite expansion' approach for additive functionals also gives a (semi)-explicit representation for the limiting variance, it has proved difficult to verify its positivity in practice. A representative example is \cite{el16} on the Euler characteristic, where the limiting variance is only shown to be positive for levels $\ell$ such that $H_d(\ell) \neq 0$, where $H_d$ is the degree-$d$ Hermite polynomial (which has $d$ zeros). Since our approach also works for additive functionals such as the Euler characteristic, we believe it gives a more tractable route to establishing strict positivity at \textit{all} levels.

\begin{remark}
The condition that $\int q(x) dx > 0$ is equivalent to either $\int K(x) dx \neq 0$ or $\rho(0) > 0$, where $\rho = \mathcal{F}^{-1}[K]$ is the spectral density of the field. While we do not expect this condition to be necessary, it is quite natural given that our proof generates fluctuations in the component count by exploiting level shifts. Indeed, this is precisely the condition under which level shifts can be well-approximated in the Cameron-Martin space of the field. An identical condition appeared in our previous study of fluctuations of the component count~\cite{bmm22}.
\end{remark}

\subsection{Third moment bounds}
A major ingredient in the proof of Theorem \ref{t:clt} is a third moment bound on the number of critical points of the field inside a compact domain; this extends existing results in dimension $d=1$ \cite{bel66,cuz75,cuz78}, and also \textit{second} moment bounds valid in all dimensions \cite{eli85,ef16}. See also \cite{mv89} for related results on third moment bounds for zeros of Gaussian vector fields, although these do not apply to critical points. Since we believe this bound may be useful in other applications, we highlight it here.

We establish this result under much more general conditions than Theorem \ref{t:clt}, and in particular we do not require any assumption on the decay of $K$. Recall that the \textit{spectral measure} of $f$ is the finite measure $\mu$ on $\R^d$ such that $K = \mathcal{F}[\mu]$.

\begin{assumption} 
\label{a:gen} $f$ is $C^4$-smooth and the support of $\mu$ contains an open set.
\end{assumption}

The condition on the support of $\mu$ is easily verified for short-range correlated fields; in particular Assumption \ref{a:clt} implies Assumption \ref{a:gen} (see Lemma \ref{l:wnr} for details).

\begin{theorem}[Third moment bound for critical points]
\label{t:tmb}
Let $\tau > 0$ and let $p \in C^4(\R^d)$ be a deterministic function such that $\|p\|_{C^4(\R^d)} \le \tau$. Suppose Assumption \ref{a:gen} holds and let $N_c(R)$ denote the number of critical points of $f+p$ contained in $\Lambda_R$. Then there exists a $c > 0$ (depending only on $f$ and $\tau$) such that, for $R \ge 1$,
\[ \E[N_c(R)^3] \le c R^{3d} .\]
\end{theorem}

We allow for the addition of the smooth function $p$ primarily because it is needed in the proof of Theorem \ref{t:posvar}, although we believe it to be of independent interest (see \cite{cuz78} for similar results in the $d=1$ case). The dependence on $\|p\|_{C^4(\R^d)}$ can be quantified, see Remark~\ref{r:const}.

Since $N_\star(R,\ell) \le N_c(R)$, an immediate consequence of Theorem \ref{t:tmb} is a third moment bound on the component count:

\begin{corollary}[Third moment bound for the component count]
\label{c:tmb}
Suppose that Assumption~\ref{a:gen} holds, then there exists a $c > 0$ such that, for $\star \in \{\mathrm{ES},\mathrm{LS}\}$, $\ell \in \R$, and $R \ge 1$, 
\[ \E[ N_\star(R,\ell)^3 ] \le c R^{3d} .\]
\end{corollary}

\begin{remark}
\label{r:nondegen}
The assumption that $\mu$ has an open set in its support can be weakened considerably; we use it only to guarantee that the following Gaussian vectors are non-degenerate for all distinct $x,y \in \R^d\setminus\{0\}$ and linearly independent vectors $v,w \in \mathbb{S}^{d-1}$:
\begin{enumerate}
    \item $(\nabla f(0), \nabla f(x), \nabla f(y))$;
     \item $(\nabla f(0), \nabla^2 f(0), \nabla f(x))$;   
    \item $ (\nabla f(0), \partial_v \nabla f(0), \partial^2_v \nabla f(0) ) $;
     \item $ (\nabla f(0),\nabla^2 f(0),\partial_v^2\partial_w f(0),\partial_v\partial_w^2f(0)) $.
\end{enumerate}
We expect that these are non-degenerate for a much wider class of stationary Gaussian fields (e.g.\ monochromatic random waves, for which $\mu$ is supported on the sphere $\mathbb{S}^{d-1}$), which would expand the scope of Theorem \ref{t:tmb}.
\end{remark}

\subsection{Outline of the paper}
In Section \ref{s:stability} we undertake a preliminary study of the stability properties of the component count. In Section \ref{s:clt} we establish the CLT stated in Theorem \ref{t:clt} and also the positivity of the limiting variance in Theorem \ref{t:posvar}. In Section \ref{s:dbcp} we study the density of critical points, and in particular prove the third moment bound in Theorem \ref{t:tmb}. Finally in the appendix we collect some technical statements about Gaussian fields and prove a topological lemma which was used in Section \ref{s:stability}.

\subsection{Acknowledgements}
The authors thank Naomi Feldheim and Rapha\"{e}l Lachi\`{e}ze-Rey for referring us to \cite{kl06} and \cite{mv89} respectively. We also thank an anonymous referee for referring us to \cite{al21}, as well as pointing out some omissions in the proof of Theorem~\ref{t:clt_gen} in a previous version of this manuscript. The second author was supported by the European Research Council (ERC) Advanced Grant QFPROBA (grant number 741487). The third author was supported by the Australian Research Council (ARC) Discovery Early Career Researcher Award DE200101467, and also acknowledges the hospitality of the Statistical Laboratory, University of Cambridge, where part of this work was carried out.

\medskip

\section{Stability of the component count}
\label{s:stability}

In this section we study the stability of the component count under perturbations. The resulting estimates will play an important role in the proof of the CLT in Section \ref{s:clt}.

\subsection{Stratified domains and critical points}
 
 A \textit{box} is any set of the form $\mathcal{R} = [a_1, b_1] \times \cdots \times [a_d, b_d] \subset \R^d$, for finite $a_i \le b_i$. Rather than work in the fullest possible generality, we restrict our attention to the stability properties of sets $D \subset \R^d$ of the form $D = D(\mathcal{R};V) = \mathcal{R} \cap ( \cup_{v \in V} B_v)$, where $\mathcal{R}$ is a box, $V \subset \mathbb{Z}^d$ is a non-empty finite subset, and $B_v = v + [0,1]^d$ denotes the translated unit cube. We refer to such sets $D$ as \textit{domains}, although we emphasise that they need not be connected. In particular $\Lambda_R = [-R,R]^d$ is of this form.

We may view any such domain $D$ as a \textit{stratified set} as follows. Recall that every box has a \textit{canonical stratification}, that is, a partition into the finite collection $\mathcal{F} = (F_i)$ of its open faces of all dimensions $0 \le m \le d$. For each cube $B_v$, $v \in \Z^d$, we denote by $\mathcal{F}_{\mathcal{R};v}$ the canonical stratification of $\mathcal{R} \cap B_v$. Then $\mathcal{F}_\mathcal{R;V} = \cup_{v \in V} \mathcal{F}_{\mathcal{R};v}$ defines a partition of $D = D( \mathcal{R};V)$, which we refer to as its \textit{stratification}.

A \textit{stratified domain} will be any $D = D(\mathcal{R};V)$ equipped with the stratification $\mathcal{F} = \mathcal{F}_{\mathcal{R};V}$.  We will occasionally need to distinguish strata of dimension $m=0$ (i.e.\ the vertices of ${\mathcal{R} \cap B_v}$, $v \in \Z^d$), which we denote by $\mathcal{F}_0 \subset \mathcal{F}$. Note that, for any stratified domain $D$ and any $v \in \Z^d$, the restriction $D \cap B_v$ may also be considered as a stratified domain with stratification $\mathcal{F}_{\mathcal{R};v}$. See Figure \ref{fig:Strat} for an example of a stratified domain in $d=2$.

\begin{figure}[ht]
    \centering
    \begin{tikzpicture}
\draw[black!20,step=0.5] (0.25,0.25) grid (7.75,3.75);
\node[] at (8,3.5)  {$\mathbb{Z}^2$};
\draw[gray, dashed] (0.75,0.75) rectangle (7.25,3.25);
\draw[gray] (7.3,2.75) -- (8,2.5);
\node[right] at (8,2.5) {$\mathcal{R}$};
\begin{scope}
\clip (0.75,0.75)--(2,0.75)--(2,1.5)--(3,1.5)--(3,3.25)--(2,3.25)--(2,3)--(1,3)--(1,2)--(1.5,2)--(1.5,2.5)--(2.5,2.5)--(2.5,2)--(1.5,2)--(1.5,1.5)--(0.75,1.5)--cycle;
\fill[gray!20] (0.75,0.75) rectangle (7.25,3.25);
\draw[thick,step=0.5] (0.25,0.25) grid (7.75,3.75);
\end{scope}
\draw[thick](0.75,0.75)--(2,0.75)--(2,1.5)--(3,1.5)--(3,3.25)--(2,3.25)--(2,3)--(1,3)--(1,2)--(1.5,2)--(1.5,2.5)--(2.5,2.5)--(2.5,2)--(1.5,2)--(1.5,1.5)--(0.75,1.5)--cycle;
\foreach \x in {(0.75,0.75),(1,0.75),(1.5,0.75),(2,0.75),(0.75,1),(1,1),(1.5,1),(2,1),(2,1.5),(3,1.5),(3,3.25),(2,3.25),(2,3),(1,3),(1,2),(1.5,2),(1.5,2.5),(2.5,2.5),(2.5,2),(1.5,2),(1.5,1.5),(0.75,1.5),(1,1.5),(2.5,1.5),(3,2),(3,2.5),(3,3),(2.5,3),(2.5,3.25),(2,2),(2,2.5),(1.5,3),(1,2.5)} {
\fill \x circle (1.5pt);
}

\begin{scope}
\clip (4,1)--(6,1)--(6,2)--(6.5,2)--(6.5,0.75)--(7.25,0.75)--(7.25,2.5)--(5.5,2.5)--(5.5,3)--(4.5,3)--(4.5,1.5)--(4,1.5)--cycle;
\fill[gray!20] (0.75,0.75) rectangle (7.25,3.25);
\draw[thick,step=0.5] (0.25,0.25) grid (7.75,3.75);
\end{scope}
\draw[thick] (4,1)--(6,1)--(6,2)--(6.5,2)--(6.5,0.75)--(7.25,0.75)--(7.25,2.5)--(5.5,2.5)--(5.5,3)--(4.5,3)--(4.5,1.5)--(4,1.5)--cycle;
\foreach \x in{4.5,5,...,7,7.25} {
\foreach \y in {1,1.5,2,2.5} {
\fill (\x,\y) circle (1.5pt);
}
}
\foreach \x in {(4,1),(4,1.5), (4.5,3),(5,3),(5.5,3),(6.5,0.75),(7,0.75),(7.25,0.75)} {
\fill \x circle (1.5pt);
}

\end{tikzpicture}
    \caption{An example of a stratified domain in $d=2$; the dashed lines show the boundary of $\mathcal{R}$ while the shaded region, thick lines and circles show the stratification of a domain $D$.}
    \label{fig:Strat}
\end{figure}
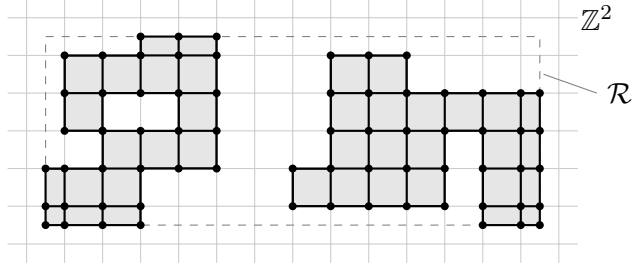

Let $D$ be a stratified domain and let $U$ be a compact set which contains an open neighbourhood of $D$. For $x \in D$ and $g \in C^1(U)$, $\nabla_F g(x)$ denotes the derivative of $g$ restricted to the unique stratum $F$ containing $x$. A \textit{(stratified) critical point} of $g$ is a point $x \in F$ such that $\nabla_F g(x) = 0$. The \textit{level} of this critical point is the value $g(x)$. By convention, all $x$ belonging to strata of dimension $m=0$ are considered critical points. The number of stratified critical points is denoted by $N_c(D,g)$. 

\subsection{Stability of the component count}
We next study the stability of the component count on stratified domains under small perturbations. In the following section we will apply these estimates to Gaussian fields.

Extending our previous notation, for a function $g \in C^1(U)$ we let $N_\star(D,g,\ell)$ denote the component count of $g$ at level $\ell$ inside $D$, i.e.\ the number of connected components of $\{g \ge \ell\}$ (if $\star=\mathrm{ES}$) or $\{g = \ell\}$ (if $\star = \mathrm{LS}$) which intersect $D$ but not its boundary.

For a pair of functions $(g,p) \in C^1(U) \times C^1(U)$, we say that $(g,p)$ is \textit{stable} (on $D$ at level $\ell$) if the following holds (and \textit{unstable} otherwise): \begin{itemize}
    \item For every $x \in F \in \mathcal{F} \setminus \mathcal{F}_0$, either
\begin{equation}\label{e:stable}
\lvert g(x)-\ell\rvert\geq 2\lvert p(x)\rvert \quad \text{or} \quad |\nabla_F g(x)|\geq 2|\nabla_F p(x)| ;
\end{equation}
\item For every $x \in F \in \mathcal{F}_0$,
\[  \lvert g(x)-\ell\rvert\geq 2\lvert p(x)\rvert . \]
\end{itemize}
This notion of stability implies that the component count is unchanged under perturbation:

\begin{lemma}
\label{l:Topological_stability}
Let $g,p \in C^2(U)$. Then if $(g,p)$ is stable (on $D$ at level $\ell$), 
\begin{displaymath}
N_\star(D,g,\ell)=N_\star(D,g+p,\ell) .
\end{displaymath}
\end{lemma}
The proof of this result is given in Appendix~\ref{s:top}. To explain the intuition, consider the interpolation $g + t p$ for $t \in [0,1]$. The pair $(g,p)$ being stable implies that there are no values of $t$ at which $g + t p$ has a critical point at level $\ell$. In that case the level sets $\{g + t p = \ell\}$ deform continuously as $t$ varies, and the component count does not change. This argument also applies to excursion sets since these have level sets as their boundaries.

A consequence is that we can bound the \textit{change} in the component count under perturbation by the total number of stratified critical points in unstable cubes: 

\begin{lemma}\label{l:stab3}
Let $g,p \in C^2(U)$. Then
\begin{displaymath}
\lvert N_\star(D,g,\ell)-N_\star(D,g+p,\ell)\rvert \leq \sum_{v\in\mathcal{U}}\left(N_c(D\cap B_v,g) + N_c(D\cap B_v,g+p)\right),
\end{displaymath}
where 
\[ \mathcal{U}:=\{v\in V  \, : \, (g,p)\text{ is unstable on }D \cap B_v \text{ at level $\ell$}\}. \]
\end{lemma}
\begin{proof}
Define $D' = D \cap (\cup_{v \in \mathcal{U}}B_v)$. Each excursion/level component inside $D$ is either contained in $D\setminus D'$ or else intersects $D'$. The number of components intersecting $D'$ is dominated by the number of stratified critical points in $D'$, and hence
\begin{displaymath}
\lvert N_\star(D,g,\ell)-N_\star(D\setminus D',g,\ell)\rvert\leq N_c(D',g)\leq \sum_{v\in\mathcal{U}} N_c(D \cap B_v,g).
\end{displaymath}
The same equation holds if $g$ is replaced by $g+p$. By Lemma~\ref{l:Topological_stability} we know that
\begin{displaymath}
N_\star(D\setminus D^\prime,g,\ell)=N_\star(D \setminus D',g+p,\ell).
\end{displaymath}
Combining these observations using the triangle inequality proves the result.
\end{proof}

\begin{remark}
One might wonder why we do not simply define $(g,p)$ to be stable if the component counts of $g$ and $g+p$ are the same. The advantage of our definition is its \textit{additivity}: if $D_1$ and $D_2$ are stratified domains and $(g, p)$ is stable on both $D_1$ and $D_2$, then $(g,p)$ is also stable on $D_1 \cup D_2$. This is not true in general for stability in the sense of the component~count.
\end{remark}

\subsection{Application to Gaussian fields}
We now give a quantitative estimate of stability when a $C^2$-smooth stationary Gaussian field is perturbed by a deterministic $C^1$ function, or more generally, a $C^1$-smooth Gaussian field, not necessarily independent of $f$.

\begin{lemma}
\label{l:quant_stab}
Let $f$ be a stationary $C^2$-smooth Gaussian field such that $(f(0),\nabla f(0))$ is non-degenerate. Let $v\in\Z^d$, let $D$ be a stratified domain which is a subset of $B_v$ and let $U$ be a compact set which contains an open neighbourhood of $D$. Then for every $\eps > 0$ there is a $c_\epsilon > 0$, independent of $D$, such that:
\begin{enumerate}
\item For all $\ell \in \R$ and every $p\in C^1(U)$,
\begin{displaymath}
    \P \big( \text{$(f, p)$ is unstable (on $D$ at level $\ell$)} \big) \le  c_\epsilon   \|p\|_{C^1(U)}^{1-\eps} .
\end{displaymath}
\item For all $\ell \in \R$ and every $C^1$-smooth Gaussian field $p$ on $U$,
\[ \P \big( \text{$(f, p)$ is unstable (on $D$ at level $\ell$)} \big) \le  c_\epsilon \inf_{\tau>M_1 + c \sqrt{M_2}} \Big( \tau^{1-\eps} + e^{-\frac{(\tau-M_1 -c\sqrt{M_2} )^2}{2M_2}} \Big) ,\]
where $c>0$ depends only on $U$,
\[ M_1 =  \|\E p\|_{C^1(U)} \quad \text{and} \quad M_2 =  \|\Cov[p]\|_{U, 1,1} = \sup_{x,y \in U} \sup_{|\alpha|, |\gamma| \le 1} \big| \partial_x^\alpha \partial_y^\gamma \Cov[p(x),p(y)] \big| .\]
\end{enumerate}
\end{lemma}

\begin{proof}
By the definition of stability and the union bound, for any $\tau \ge 0$, the probability that $(f,p)$ is unstable is bounded above by
\begin{align}\label{e:Unstable}
 \sum_{F \in \mathcal{F}}\P\Big(\inf_{x \in F}\max\{\lvert f(x)-\ell\rvert,|\nabla_F f(x)|\}<2\tau\Big)+\P\Big(\sup_{x \in F}\max\{\lvert p(x)\rvert,|\nabla_F p(x)|\}>\tau\Big) 
\end{align}
(ignoring the $\nabla_F$ terms if $F \in \mathcal{F}_0$). Since we assume that $D \subseteq B_v$, by stationarity and monotonicity it suffices to control the terms in \eqref{e:Unstable} in the case that $D = B_0$.

The first term in \eqref{e:Unstable} can be bounded by a quantitative version of Bulinskaya's lemma. Specifically \cite[Lemma~7]{ns16} states that, for any $\eps\in(0,1)$, there exists $c_\eps>0$ such that, for all $\tau>0$,
\begin{equation}\label{e:Unstable2}
\P\Big(\inf_{x \in F}\max\{\lvert f(x) -\ell\rvert,|\nabla_F f(x) |\}<2\tau\Big)  \le c_\eps \tau^{1-\eps}.
\end{equation}
We note that the exponent $1-\eps$ is not given in the statement of \cite[Lemma~7]{ns16} but follows immediately from the final inequality in its proof. Then the first statement of the lemma follows by setting $\tau = \|p\|_{C^1(U)}$ in \eqref{e:Unstable}.

For the second statement, we need to bound the second term in \eqref{e:Unstable}. First let us assume that $p$ is centred. For each stratum $F \in \mathcal{F}$ we define
\begin{displaymath}
\|p\|_{F,1}=\sup_{x\in F}\sup_{\lvert\alpha\rvert\leq 1}\left\lvert\partial^\alpha p|_F(x) \right\rvert \quad\text{and} \quad\sigma_F^2=\sup_{x\in F}\sup_{\lvert\alpha\rvert\leq 1}\Var\left[\partial^\alpha p|_F(x)\right].
\end{displaymath}
By the Borell-TIS inequality \cite[Theorem~2.1.1]{at07}, for any $\tau \ge \E\left[\|p\|_{F,1}\right]$
\begin{equation}\label{e:Unstable3}
\P\big( \|p\|_{F,1}>\tau \big)\leq e^{ -(\tau-\E\left[\|p\|_{F,1}\right])^2 / (2\sigma_F^2) }.
\end{equation}
By the quantified Kolmogorov's theorem \cite[Section~A.9]{ns16},
\begin{align*}
\E\left[\|p\|_{F,1}\right] \leq c \sqrt{M_2} 
\end{align*}
where $c > 0$ depends only on $U$. Since we also have
\begin{displaymath}
\sigma_F^2=\sup_{x\in F}\sup_{\lvert\alpha\rvert\leq 1}\left\lvert\partial^\alpha_x\partial^\alpha_y \Cov[p(x),p(y)]\big|_{y=x}\right\rvert\leq M_2,
\end{displaymath}
combining \eqref{e:Unstable2} and \eqref{e:Unstable3} gives the second statement.

In the case that $p$ is not centred, for all $\tau>0$
\begin{multline*}
\P\Big(\sup_{x \in F}\max\{\lvert p(x)\rvert,|\nabla_F p(x)|\}>\tau\Big)\\
\leq \P\Big(\sup_{x \in F}\max\{\lvert p(x)-\E[p(x)]\rvert,|\nabla_F p(x)-\E[\nabla_F p(x)]|\}>\tau - \|\E p\|_{C^1(U)}\Big).
\end{multline*}
Therefore the above arguments are valid on replacing $\tau$ with $\tau - \|\E p\|_{C^1(U)}$.
\end{proof}

\smallskip

\section{Proof of the CLT}
\label{s:clt}
In this section we give the proof of the CLT in Theorem~\ref{t:clt}, and also of the positivity of the limiting variance in Theorem \ref{t:posvar}. 

\subsection{Martingale CLT for lexicographic arrays}
The basis of the proof of Theorem \ref{t:clt} is a classical CLT for martingale arrays which we now describe. This generalises the approach of Penrose in \cite{pen01}.

We say that a collection of random variables $S_{n,i}$ and $\sigma$-algebras $\mathcal{F}_{n,i}$, indexed by $n\in\N$ and $i=1,\dots,k_n$, form a \emph{martingale array} if for all $n$ and $i$
\begin{displaymath}
\E[S_{n,i+1}|\mathcal{F}_{n,i}]=S_{n,i}\quad\text{and}\quad \mathcal{F}_{n,i}\subseteq\mathcal{F}_{n,i+1}.
\end{displaymath}
We say that the array is \textit{mean-zero} if $\E[S_{n,i}]=0$ for each $n,i$, and \textit{square-integrable} if $\sup_{i}\E[S_{n,i}^2]<\infty$ for each $n$ and we define the \textit{differences} of the array as $U_{n,i}:=S_{n,i}-S_{n,i-1}$ for $i=1,\dots,k_n$ (setting $S_{n,0}=0$ by convention).

\begin{theorem}[{\cite[Theorem~3.2]{hh80}}]\label{t:Basic_MCLT}
Let $\{S_{n,i},\mathcal{F}_{n,i}:1\leq i\leq k_n, n\in\N\}$ be a mean-zero square-integrable martingale array with differences $U_{n,i}$. Suppose that
\begin{align}
&\max_{i=1,\dots,k_n}\lvert U_{n,i}\rvert\xrightarrow{p}0\quad\text{as }n\to\infty,\\
&\sup_n\E\left[\max_{i=1,\dots,k_n}U_{n,i}^2\right]<\infty,\\
&\sum_{i=1}^{k_n}U_{n,i}^2\xrightarrow{p}\eta^2\in[0,\infty)\quad\text{as }n\to\infty. \label{e:MCLT3_basic}
\end{align}
Then $S_{n,k_n}\xrightarrow{d}Z,$ where $Z\sim\mathcal{N}(0,\eta^2)$.
\end{theorem}

In \cite{pen01} Penrose applied Theorem \ref{t:Basic_MCLT} to prove a CLT for the number of clusters in classical percolation. There is a technical difficulty in extending this argument to our setting: the number of percolation clusters in a box depends on a finite number of random variables, whereas the component count for a Gaussian field depends on the white noise throughout $\R^d$ (or equivalently on restrictions of the white noise to unit cubes indexed by $\Z^d$). We therefore require a version of this result for infinite martingale arrays (i.e.\ we need to allow $k_n$ to be infinite). More precisely we wish to apply this result in the case that $i$ is indexed by $\Z^d$ with the standard lexicographic ordering. The role of this particular ordering will become clear later on, the important fact is that it is preserved by shifts of $\Z^d$. 

To this end we make some definitions. We say that a collection of random variables $S_{n,v}$ and $\sigma$-algebras $\mathcal{F}_{n,v}$, indexed by $n\in\N$ and $v\in\Z^d$, form a \emph{lexicographic martingale array} if for all $n\in\N$ and all $v\preceq w$ (where $\preceq$ denotes the lexicographic order) we have
\begin{displaymath}
\E[S_{n,w}|\mathcal{F}_{n,v}]=S_{n,v}\quad\text{and}\quad\mathcal{F}_{n,v}\subseteq\mathcal{F}_{n,w}.
\end{displaymath}
The array is \textit{mean-zero} if $\E[S_{n,v}] = 0$, and \textit{square-integrable} if $\sup_{v}\E[S_{n,v}^2]<\infty$ for each $n$.

We say that a sequence of points in $\Z^d$ \textit{converges to $\pm\infty^*$ (in the lexicographic ordering)} if each coordinate of the points tends to $\pm\infty$. By the backwards martingale convergence theorem, for each $n$, $S_{n,v}$ converges almost surely to some integrable limit as $v$ tends to $-\infty^*$. Therefore we assume without loss of generality that this limit is zero for each $n$. 

We say that the array is \emph{regular at infinity} if the following holds: for each $n\in\N$, ${i\in\{1,\dots,d-1\}}$ and $(v_1,\dots,v_i)\in\Z^i$
\begin{displaymath}
\lim_{v_{i+1},\dots,v_d\to\infty}S_{n,v}=\lim_{v_{i+1},\dots,v_d\to-\infty}S_{n,v^\prime}
\end{displaymath}
where $v=(v_1,\dots,v_d)$ and  $v^\prime=(v_1,\dots,v_{i-1},v_i+1,v_{i+1}\dots,v_d)$. (Note that both of the above limits exist, almost surely and in $L^2$, by the $L^p$ martingale convergence theorems for forward/reverse martingales.) For $v\in\Z^d$, let $v^-$ denote the previous element in the lexicographic ordering of $\Z^d$, and define the \textit{differences} of the array as ${U_{n,v}:=S_{n,v}-S_{n,v^-}}$. A simple argument iterating over the coordinates of $v$ shows that being regular at infinity allows us to write\begin{displaymath}
S_{n,v}-S_{n,w}=\sum_{w\prec k\preceq v}U_{n,k}
\end{displaymath}
for any $w\preceq v$, where the limit implied by this summation holds almost surely and in $L^2$. Note that a priori we do not know that this sum converges absolutely, so the order of summation should be taken in a way that is consistent with the lexicographic order. However we only work in the setting that $\E[\sum_{v\in\Z^d}\lvert U_{n,v}\rvert^p]<\infty$ for $p\in\{1,2\}$, and so we may ignore this subtlety. Then using the $L^p$ martingale convergence theorems for forward/reverse martingales, we see that (for an array which is square-integrable and regular at infinity)
\begin{equation}\label{e:OrthogonalMart}
S_{n,\infty^*}:=\lim_{v\to\infty^*}S_{n,v}=\sum_{k\in\Z^d}U_{n,k}
\end{equation}
where the limit holds almost surely and in $L^2$.

\begin{theorem}\label{t:Lex_MCLT}
Let $\{S_{n,v},\mathcal{F}_{n,v}:v\in\Z^d, n\in\N\}$ be a mean-zero square-integrable lexicographic martingale array which is regular at infinity. Suppose that
\begin{align}
&\sup_{v\in\Z^d}\lvert U_{n,v}\rvert\xrightarrow{p}0\quad\text{as }n\to\infty\label{e:MCLT1}\\
&\sup_n\E\Big[\sup_{v\in\Z^d}U_{n,v}^2\Big]<\infty\label{e:MCLT2}\\
&\sum_{v\in\Z^d}U_{n,v}^2\xrightarrow{L^1}\eta^2\in[0,\infty)\quad\text{as }n\to\infty\label{e:MCLT3}\\
&\E\Big[\sum_{v\in\Z^d}\lvert U_{n,v}\rvert\Big]<\infty\quad\text{for all }n\in\N.\label{e:MCLT4}
\end{align}
Then $\Var[S_{n,\infty^*}]\to\eta^2$ and $S_{n,\infty^*}\xrightarrow{d}Z$ where $Z\sim\mathcal{N}(0,\eta^2)$.
\end{theorem}
\begin{remark}
Compared to Theorem \ref{t:Basic_MCLT}, as well as holding for lexicographic arrays Theorem~\ref{t:Lex_MCLT} also strengthens the mode of convergence by adding the summability condition \eqref{e:MCLT4} and assuming $L^1$ convergence in \eqref{e:MCLT3} rather than convergence in probability as in \eqref{e:MCLT3_basic}. 
\end{remark}
\begin{proof}
Using orthogonality of martingale increments and \eqref{e:OrthogonalMart}
\begin{displaymath}
\Var\Big[S_{n,\infty^*}\Big]=\E \Big[\sum_{v\in\Z^d}U_{n,v}^2 \Big],
\end{displaymath}
and hence \eqref{e:MCLT3} implies that $\Var[S_{n,\infty^*}]\to\eta^2$, proving the first part of our result. Moreover combining this with \eqref{e:MCLT4} we can find a sequence of points $a_n\in\N$ tending to infinity such that as $n\to\infty$
\begin{equation}\label{e:LMCLT1}
\sup_{p=1,2}\E\bigg[\sum_{v\notin \Lambda_{a_n-1}}\lvert U_{n,v}\rvert^p\bigg]\to 0.
\end{equation}
Recalling that martingale increments are orthogonal, this implies that $\sum_{v\notin \Lambda_{a_n-1}}U_{n,v}$ converges to zero in $L^2$ (and hence in probability).

We now define a finite martingale array by restricting $S_{n,v}$ to the box $\Lambda_{a_n}$. Specifically for each $n$ we let $v_1,v_2,\dots,v_{k_n}$ denote the elements of $\Lambda_{a_n}\cap\Z^d$ ordered lexicographically and we define
\begin{displaymath}
T_{n,j}:=S_{n,v_j},\quad\mathcal{G}_{n,j}:=\mathcal{F}_{n,v_j},\quad V_{n,j}:=T_{n,j}-T_{n,j-1}=\sum_{v_{j-1}\prec v\preceq v_j}U_{n,v}\quad\text{for }j=1,\dots,k_n
\end{displaymath}
with the convention that $v_0=-\infty^*$. We wish to apply Theorem~\ref{t:Basic_MCLT} to this construction. First we consider the differences $V_{n,j}$ in terms of $U_{n,v}$. If the point preceding $v_j$ lexicographically is contained in $\Lambda_{a_n}$ (i.e.\ if $v_j^-=v_{j-1}$) then $V_{n,j}=U_{n,v}$ for some $v\in\Lambda_{a_n}$. Otherwise we have $v\notin\Lambda_{a_n}$ for all $v$ such that $v_{j-1}\prec v\prec v_j$. Considering these two cases, we see that
\begin{equation*}\label{e:LMCLT2}
\max_{j=1,\dots,k_n}\lvert V_{n,j}\rvert\leq\max_{v\in\Lambda_{a_n}}\lvert U_{n,v}\rvert+\sum_{j:v_j^-\notin\Lambda_{a_n}}\sum_{v_{j-1}\prec v\preceq v_j}\lvert U_{n,v}\rvert\leq\sup_{v\in\Z^d}\lvert U_{n,v}\rvert+\sum_{v\notin\Lambda_{a_n-1}}\lvert U_{n,v}\rvert.
\end{equation*}
By \eqref{e:MCLT1} and \eqref{e:LMCLT1}, this converges to zero in probability, verifying the first condition of Theorem~\ref{t:Basic_MCLT}.

By the same reasoning
\begin{equation*}
\max_{j=1,\dots,k_n} V_{n,j}^2\leq\max_{v\in\Lambda_{a_n}} U_{n,v}^2+\sum_{j:v_j^-\notin\Lambda_{a_n}}\bigg(\sum_{v_{j-1}\prec v\preceq v_j} U_{n,v}\bigg)^2.
\end{equation*}
Taking expectations, using orthogonality of martingale increments, \eqref{e:MCLT2} and \eqref{e:LMCLT1}, we see that the second condition of Theorem~\ref{t:Basic_MCLT} also holds.

We also have
\begin{equation*}\label{e:LMCLT3}
\E\bigg\lvert\sum_{j=1}^{k_n}V_{n,j}^2-\sum_{v: v,v^-\in\Lambda_{a_n}}U_{n,v}^2\bigg\rvert\leq\E\bigg[\sum_{j:v_j^-\notin\Lambda_{a_n}}\bigg(\sum_{v_{j-1}\prec v\preceq v_j} U_{n,v}\bigg)^2\bigg]\leq\sum_{v\notin\Lambda_{a_n-1}}\E[U_{n,v}^2],
\end{equation*}
which converges to zero by \eqref{e:LMCLT1}. Hence by \eqref{e:MCLT3} (and applying \eqref{e:LMCLT1} once more) we verify the third and final condition of Theorem~\ref{t:Basic_MCLT}. We therefore deduce that
\begin{displaymath}
\sum_{j=1}^{k_n}V_{n,j}\xrightarrow{d}Z
\end{displaymath}
where $Z\sim\mathcal{N}(0,\eta^2)$. Finally we observe that by the two types of expression for $V_{n,j}$ in terms of $U_{n,v}$ described above
\begin{equation*}\label{e:LMCLT4}
\bigg\lvert\sum_{j=1}^{k_n}V_{n,j}-S_{n,\infty^*}\bigg\rvert=\bigg\lvert\sum_{j=1}^{k_n}V_{n,j}-\sum_{v\in\Z^d}U_{n,v}\bigg\rvert\leq\sum_{v\notin\Lambda_{a_n}}\lvert U_{n,v}\rvert.
\end{equation*}
The right hand side converges to zero in $L^1$ by \eqref{e:LMCLT1} and so we see that $S_{n,\infty^*}\xrightarrow{d}Z$ as required.
\end{proof}

\subsection{Application to the component count}
We can now outline our proof of Theorem~\ref{t:clt}. In fact we will prove a version of this result which holds for the number of components contained in slightly more general domains (i.e.\ not just the cube $\Lambda_R=[-R,R]^d$). 

 We say that $D\subset \R^d$ is a \textit{box-domain} if it is a box $D =  [a_1,b_1]\times\cdots\times[a_d,b_d] $ considered as a stratified domain in the sense of Section~\ref{s:stability}. We write $|D|$ for the volume of a box-domain, and $\mathrm{asp}(D) = \min_i |b_i-a_i| / \max_i |b_i - a_i|$ for its \textit{aspect ratio}, i.e.\ the ratio of its shortest to longest side lengths. We say that a sequence of box-domains $D_n$ \textit{converges to} $\R^d$ (written $D_n \to \R^d$) if $\cup_n\cap_{m>n}D_m=\R^d$ and $\inf_n \mathrm{asp}(D_n) > 0$.

Henceforth we suppose that $f$ satisfies Assumption~\ref{a:clt}, and fix $\star \in \{\mathrm{ES},\mathrm{LS}\}$ and $\ell \in \R$. Recall from Section \ref{s:stability} that $N_\star(D)=N_\star(D,f,\ell)$ denotes the component count for $f$ at level~$\ell$ inside $D$. Our desired CLT now takes the following form:

\begin{theorem}[CLT for the component count]
\label{t:clt_gen}
There exists a $\sigma^2 \ge 0$ such that, for every sequence of box-domains $D_n \to \R^d$, as $n\to\infty$
\[ \frac{\Var[N_\star(D_n)]}{\lvert D_n\rvert}\to\sigma^2 \quad \text{and} \quad 
\frac{N_\star(D_n) -  \E[N_\star(D_n)]}{\lvert D_n\rvert^{1/2}}\xrightarrow{d} \sigma Z , \]
where $Z$ is a standard normal random variable. A representation for $\sigma^2$ is given in \eqref{e:sigma}.
\end{theorem}

\begin{proof}[Proof of Theorem~\ref{t:clt}]
For a sequence $R_n \to \infty$, take $D_n=[-R_n,R_n]^d$ in Theorem~\ref{t:clt_gen}.
\end{proof}

Theorem \ref{t:clt_gen} is proven in the following way. Recall that $f$ may be represented as $f=q\ast W$ where $W$ is the white noise on $\R^d$. For $v\in\Z^d$, recall that $B_v=v+[0,1]^d$, and write $W_v$ for the restriction of $W$ to the cube $B_v$.  Fix a sequence of box-domains $D_n \to \R^d$, and define for $v\in\Z^d$
\begin{align}\label{e:LexMarRep}
S_{n,v}=\frac{\E\left[N_\star(D_n)\;\middle|\;\mathcal{F}_{n,v}\right]-\E[N_\star(D_n)]}{\left\lvert D_n\right\rvert^{1/2}}\quad\text{where}\quad\mathcal{F}_{n,v}:=\mathcal{F}_v:=\sigma\left(W_u\;|\;u\preceq v\right),
\end{align}
recalling that $\preceq$ denotes the lexicographic order.
\begin{lemma}\label{l:RegularInf}
    Equation~\eqref{e:LexMarRep} defines a mean-zero lexicographic martingale array which is square-integrable, regular at infinity and satisfies
    \begin{equation}\label{e:MartingaleNormalisation}
        \lim_{v\to-\infty^*}S_{n,v}=0\quad\text{and}\quad
        S_{n,\infty^*}=\lim_{v\to\infty^*}S_{n,v}=\frac{N_\star(D_n)-\E[N_\star(D_n)]}{\lvert D_n\rvert^{1/2}}.
    \end{equation}
\end{lemma}
\begin{proof}
    It is clear that \eqref{e:LexMarRep} defines a lexicographic martingale array (by the tower property of conditional expectation). Square-integrability follows since $N_\star(D_n)$ has a finite second moment for each $n$. By L\'evy's downward and upward convergence theorems respectively we have
    \begin{equation}\label{e:Filtration_normalisation}
    \begin{aligned}
    \text{as }v\to-\infty^*\qquad \E[N_\star(D_n)\;|\;\mathcal{F}_{v}]\to\E\left[N_\star(D_n)\;\middle|\;\mathcal\cap_{v\in\Z^d}{F}_{v}\right]=\E[N_\star(D_n)]\\
    \text{as }v\to\infty^*\qquad \E[N_\star(D_n)\;|\;\mathcal{F}_{v}]\to\E\left[N_\star(D_n)\;\middle|\;\sigma\left(\cup_{v\in\Z^d}\mathcal{F}_{v}\right)\right]=N_\star(D_n).
    \end{aligned}
    \end{equation}
    This verifies \eqref{e:MartingaleNormalisation}.

    Finally we show that \eqref{e:LexMarRep} is regular at infinity. We fix $v_1,\dots,v_i\in\Z$ and two sequences $j^{(m)}$ and $k^{(m)}$ in $\Z^d$ such that
    \begin{align*}
        j^{(m)}&=(v_1,\dots,v_i,a^{(m)}_1,\dots,a^{(m)}_{d-i})\\
        k^{(m)}&=(v_1,\dots,v_i+1,b^{(m)}_1,\dots,b^{(m)}_{d-i})
    \end{align*}
    where $a^{(m)}_1,\dots,a^{(m)}_{d-i}\to\infty$ and $b^{(m)}_1,\dots,b^{(m)}_{d-i}\to-\infty$ with $m$. By L\'evy's upward and downward theorems, as $m\to\infty$ we have
    \begin{align*}
        \E[N_\star(D_n)\;|\;\mathcal{F}_{j^{(m)}}]\to\E[N_\star(D_n)\;|\;\mathcal{F}^-]\quad\text{and}\quad
        \E[N_\star(D_n)\;|\;\mathcal{F}_{k^{(m)}}]\to\E[N_\star(D_n)\;|\;\mathcal{F}^+]
    \end{align*}
    where
    \begin{displaymath}
        \mathcal{F}^-=\sigma\Big(\bigcup_{m\in\N}\mathcal{F}_{j^{(m)}}\Big)\quad\text{and}\quad\mathcal{F}^+=\bigcap_{m\in\N}\mathcal{F}_{k^{(m)}}.
    \end{displaymath}
    Regularity at infinity then follows if we can show that the completions of $\mathcal{F}^+$ and $\mathcal{F}^-$ coincide. Observe that $\mathcal{F}^+$ is generated by events in $\mathcal{F}^-$ together with those measurable with respect to a `tail' of independent variables. Therefore we can prove equality by generalising the proof of Kolmogorov's zero-one law. Specifically for $A\in\mathcal{F}^+$, defining
    \begin{align*}
        \mathcal{G}_m:=\sigma(W_v\;|\;k^{(m)}\preceq v\preceq k^{(1)})\quad\text{and}\quad\mathcal{G}_\infty=\sigma\Big(\bigcup_m\mathcal{G}_m\Big)
    \end{align*}
    we may apply L\'evy's upward theorem once more to see that
    \begin{equation}\label{e:LevyUpward}
        \E[\ind_A\;|\;\sigma(\mathcal{F}^-,\mathcal{G}_m)]\to \E[\ind_A\;|\;\sigma(\mathcal{F}^-,\mathcal{G}_\infty)]=\ind_A
    \end{equation}
    where the final equality follows since $\sigma(\mathcal{F}^-,\mathcal{G}_\infty)\supseteq\mathcal{F}_{k^{(1)}}\supseteq\mathcal{F}^+$. However since $\mathcal{F}^+$ is independent of $\mathcal{G}_m$ we have
    \begin{displaymath}
        \E[\ind_A\;|\;\sigma(\mathcal{F}^-,\mathcal{G}_m)]=\E[\ind_A\;|\;\mathcal{F}^-].
    \end{displaymath}
    Combined with \eqref{e:LevyUpward} we conclude that $A$ is measurable with respect to the completion of $\mathcal{F}^-$, as required.
\end{proof}

As a consequence of this lemma, Theorem~\ref{t:clt_gen} will follow from our lexicographic martingale CLT (Theorem~\ref{t:Lex_MCLT}) provided we verify conditions \eqref{e:MCLT1}--\eqref{e:MCLT4}. This verification will require several preparatory lemmas.

The first step is to find an alternative representation of the martingale differences $U_{n,v}$. Let $W^\prime$ be an independent copy of $W$ and define a new white noise
\begin{displaymath}
\widetilde{W}_v(A)=W(A\backslash B_v)+W^\prime(A\cap B_v).
\end{displaymath}
Roughly speaking $\widetilde{W}_v$ is simply $W$ after resampling independently on $B_v$. We then define $\widetilde{f}_v=q\ast\widetilde{W}_v$ and
\begin{equation}\label{e:DeltaDef}
\Delta_v(D)=N_\star(D,f,\ell)-N_\star(D,\widetilde{f}_v,\ell)
\end{equation}
i.e.\ $\Delta_v(D)$ is the change in the component count inside a domain $D$ upon resampling the white noise on the cube $B_v$. The importance of resampling comes from the following representation
\begin{equation}\label{e:ResamplingRep}
U_{n,v}=\lvert D_n\rvert^{-1/2}\E[\Delta_v(D_n)\;|\;\mathcal{F}_v]\qquad\text{a.s.,}
\end{equation}
which follows easily from the independence of the white noise on disjoint regions. Finally we define the perturbation function $p_v:\R^d\to\R$ by
\begin{displaymath}
p_v(x):=\widetilde{f}_v(x)-f(x)=\int_{B_v}q(x-u)\;d(W^\prime-W)(u).
\end{displaymath}
Note that $\widetilde{f}_v$ is equal to $f$ in law, and $p_v$ is a centred Gaussian field.

We can control $\Delta_v(D)$ by applying the stability estimates from Section \ref{s:stability}. We let $V = \{w \in \Z^d : D \cap B_w \neq \emptyset\}$ denote the indices of cubes which intersect $D$. For $w \in V$, recall the notion of stability of a pair of functions $(g, p)$ (at level $\ell$) on $D \cap B_w$, considered as a stratified domain. For $v\in\Z^d$, we define the (random) \textit{unstable set}
\begin{displaymath}
\mathcal{U}_v = \U_v(D) :=\{w\in V \;|\;(f,p_v) \text{ is unstable on $D \cap B_w$}\}.
\end{displaymath}
Recall also the exponent $\beta > 9d$ from Assumption~\ref{a:clt}, which appears in subsequent bounds.

\begin{lemma}\label{l:Prob_unstable_decay}
For each $\delta>0$ there exists $c>0$, independent of $D$, such that, for all $v\in\Z^d$,
\begin{displaymath}
\P\left(v\in \U_0\right)\leq c( 1+ \lvert v\rvert)^{-\beta+\delta} .
\end{displaymath}
\end{lemma}
\begin{proof}
By Lemma \ref{l:quant_stab}, for every $\eps > 0$ there exist $c_\eps,c > 0$, independent of $D$, such that 
\[ \P\left(v\in \U_0\right)  \le  c_\eps \inf_{\tau > c \sqrt{M_v} } \Big( \tau^{1-\eps} +  e^{-(\tau- c \sqrt{M_v})^2 /  (2 M_v)  } \Big),\]
where
\[ M_v= \sup_{x,y \in v + \Lambda_2 } \sup_{|\alpha|, |\gamma| \le 2} \big| \partial_x^\alpha \partial_y^\gamma \Cov[p_0(x),p_0(y)] \big| . \]
 By the white noise representation of $p_0$, for $|\alpha|, |\gamma| \le 2$,
\[ \left\lvert\partial^\alpha_x\partial^\gamma_y \Cov[p_0(x),p_0(y)]\right\rvert = \left\lvert  2 \int_{B_0} \partial^\alpha_xq(x-u)\partial^\gamma_yq(y-u)\;du \right\lvert, \]
where the exchange of derivative and integration is justified by the dominated convergence theorem, since $q \in C^3(\R^d)$ and $B_0$ is compact. By Assumption~\ref{a:clt} we conclude that $M_v\leq c_1 ( 1 + \lvert  v\rvert)^{-2\beta}$. Choosing $\tau= 2 c \sqrt{c_1} (1 + \lvert v\rvert )^{-\beta+\eta} \ge 2 c \sqrt{M_v}$ for some $\eta>0$, we have
\begin{displaymath}
 \P(v\in \U_0) \le c_2 \Big( ( 1 + \lvert  v\rvert)^{-(\beta-\eta)(1-\eps)}+e^{-c_3( 1 + \lvert  v\rvert)^{2\eta}} \Big)
\end{displaymath}
for some $c_2, c_3 > 0$ independent of $D$. Choosing $\eta,\eps>0$ small enough we can ensure that this expression is bounded by $c_\delta (1 + \lvert  v\rvert)^{-\beta+\delta}$ for any $\delta>0$ as required.
\end{proof}

The previous lemma, combined with the stability in Lemma \ref{l:Topological_stability}, shows that with high probability $\Delta_v(D) = 0$ if $v$ is far away from $D$. In the following lemma we control the ${(2+\eps)}$-moments of $\Delta_v(D)$; this makes essential use of the third moment bound in Theorem~\ref{t:tmb}.

\begin{lemma}[Bounded moments]\label{l:Bounded_moments}
For each $\eps\in[0,1-9d/\beta)$ the following hold:
\begin{enumerate}
    \item  There is a $c_1 > 0$ such that, for every box-domain $D$,
\begin{displaymath}
\E\left[\lvert\Delta_0(D)\rvert^{2+\eps}\right] \le c_1  .
\end{displaymath}
\item For every $a > 0$ and $\gamma < \beta(1-\eps)/3$ there is a $c_2 > 0$ such that, for every box-domain $D$ with $\mathrm{asp}(D) \ge a$, and $R \ge 1$
\[ \sum_{\mathrm{dist}(v,D)>R}\E\left[\lvert\Delta_v(D)\rvert^{2+\eps}\right]\leq c_2 \lvert D\rvert^3 R^{-\gamma}(R^d +R\lvert D\rvert^{\frac{d-1}{d}}). \]
\item For every $a > 0$ there is a $c_3 > 0$ such that, for every box-domain $D$ with $\mathrm{asp}(D) \ge a$,
\begin{displaymath}
\sum_{v\in\Z^d}\E\left[\lvert\Delta_v(D)\rvert^{2+\eps}\right]\leq c_3 \lvert D\rvert.
\end{displaymath}
\end{enumerate}
\end{lemma}
\begin{proof}
Recall that $N_c(D \cap B_v,g)$ denotes the number of stratified critical points of $g$ in $D \cap B_v$, and define $\overline{N}_c(B_v) = N_c(D \cap B_v,f) + N_c(D \cap B_v,\widetilde{f}_0)$. By Lemma~\ref{l:stab3} (applied to $g = f$ and $p = p_0)$
\begin{equation}\label{e:Triangle_ineq_2}
\lvert\Delta_0(D)\rvert\leq\sum_{v\in V}\overline{N}_c(B_v)\ind_{v\in \U_0}.
\end{equation}
In order to control the $(2+\eps)$-moment of this quantity, we use the following elementary bound which follows from the reverse Minkowski inequality: for $x_1,\dots,x_n\geq 0$ and ${p\in(0,1)}$
\begin{equation}\label{e:Reverse_Minkowski}
\Big(\sum_{i=1}^nx_i^p\Big)^{1/p}\geq(x_1^p)^{1/p}+\dots (x_n^p)^{1/p}=\sum_{i=1}^nx_i.
\end{equation}
Combining this with \eqref{e:Triangle_ineq_2} yields, for any $\eps \in (0,1)$,
\begin{align*}
\lvert\Delta_0(D)\rvert^{2+\eps}&\leq\bigg(\sum_{v\in V}\overline{N}_c\left(B_v\right)\ind_{v\in \U_0}\bigg)^{2+\eps}=\bigg(\Big(\sum_{v\in V}\overline{N}_c\left(B_v\right)\ind_{v\in \U_0}\Big)^{3}\bigg)^{\frac{2+\eps}{3}}\\
&=\bigg(\sum_{u,v,w\in V}\overline{N}_c\left(B_u\right)\overline{N}_c\left(B_v\right)\overline{N}_c\left(B_w\right)\ind_{u,v,w\in 
\U_0}\bigg)^{\frac{2+\eps}{3}}\\
&\leq\sum_{u,v,w\in V}\Big(\overline{N}_c\left(B_u\right)\overline{N}_c\left(B_v\right)\overline{N}_c\left(B_w\right)\Big)^{\frac{2+\eps}{3}}\ind_{u,v,w\in \U_0}.
\end{align*}
Taking expectations and using H\"older's inequality we have
\begin{equation}\label{e:HolderBM}
\E\left[\lvert\Delta_0(D)\rvert^{2+\eps}\right]\leq\sum_{u,v,w\in V}\prod_{i\in\{u,v,w\}}\E\Big[\overline{N}_c\left(B_i\right)^3\Big]^{\frac{2+\eps}{9}}\P(i\in \U_0)^{\frac{1-\eps}{9}}.
\end{equation}
By the third moment bound (Theorem~\ref{t:tmb}) applied to $f$ and $\widetilde{f}_0$ over each stratum of $B_i$, and by stationarity, we have that $\E[\overline{N}_c\left(B_i\right)^3]$ is uniformly bounded over $i\in\Z^d$. We therefore see that
\begin{equation}\label{e:Bounded_moments2}
\E\left[\lvert\Delta_0(D)\rvert^{2+\eps}\right]\leq c_4\bigg(\sum_{v\in V}\P\left(v\in \U_0\right)^{\frac{1-\eps}{9}}\bigg)^3\leq c_4\bigg(\sum_{v\in \Z^d}\P\left(v\in \U_0\right)^{\frac{1-\eps}{9}}\bigg)^3
\end{equation}
for a constant $c_4>0$ depending only on $f$. Lemma~\ref{l:Prob_unstable_decay} shows that this summand is bounded by $c_5(1+\lvert v\rvert)^{-(\beta-\delta)(1-\eps)/9}$ for any $\delta>0$. Choosing $\delta$ sufficiently small so that ${(\beta-\delta)(1-\epsilon)/9>d}$ (which is possible since $\eps<1-9d/\beta$) ensures that the exponent is less than $-d$ and so \eqref{e:Bounded_moments2} is bounded uniformly over $D$.

We turn to the second statement of the lemma. Let $\delta > 0$ and define $\gamma = (\beta-\delta)(1-\epsilon)/3$; without loss of generality we may assume $\delta$ is sufficiently small so that $\gamma > 3d$. By~\eqref{e:Bounded_moments2} and stationarity,
\begin{align}
\nonumber    \sum_{\mathrm{dist}(v,D)>R}\E\left[\lvert\Delta_v(D)\rvert^{2+\eps}\right]&\leq c_6\sum_{\mathrm{dist}(v,D)>R}\bigg(\sum_{w\in V}(1+\lvert v-w\rvert)^{-\gamma/3}\bigg)^3\\
 \nonumber   &\leq c_6\sum_{\mathrm{dist}(v,D)>R}\bigg(\sum_{w\in V}\mathrm{dist}(v,D)^{-\gamma/3}\bigg)^3\\
 \label{e:BddMoments}   &\leq c_6\lvert D\rvert^3\sum_{\mathrm{dist}(v,D)>R}\mathrm{dist}(v,D)^{-\gamma}.
\end{align}
We now claim that
\begin{equation}\label{e:BddMoments2}
\begin{aligned}
\sum_{\mathrm{dist}(v,D)>R}\mathrm{dist}(v,D)^{-\gamma}\leq c_7 \Big( R^{-\gamma+d}+ \lvert D\rvert^{\frac{d-1}{d}}R^{-\gamma+1} \Big) ,
\end{aligned}
\end{equation}
where $c_7> 0$ depends on $d$, $\gamma$ and $a$. Then combining \eqref{e:BddMoments} and \eqref{e:BddMoments2} establishes the second statement of the lemma. It remains to prove \eqref{e:BddMoments2}, which we do below in Lemma~\ref{l:ElementarySum}.

For the third statement, we take $R$ to be the shortest side length of $D$, which is comparable to $\lvert D\rvert^{1/d}$ by assumption, and partition $\Z^d$ into regions with distance greater/less than $R$ from~$D$. Then choose $\gamma \in (3d, \beta(1-\eps)/3)$, and combine the first two statements of the lemma.
\end{proof}

\begin{lemma}\label{l:ElementarySum}
    Let $\gamma>d$ and let $D$ be a box-domain such that $\mathrm{asp}(D)\geq a>0$ and $\lvert D\rvert\geq 1$, then there exists a constant $c>0$ depending only on $d$, $\gamma$ and $a$ such that for $R\geq 1$
    \begin{displaymath}
        \sum_{v\in\Z^d\;:\;\dist(v,D)>R}\dist(v,D)^{-\gamma}\leq c(R^{-\gamma+d}+\lvert D\rvert^{\frac{d-1}{d}}R^{-\gamma+1}).
    \end{displaymath}
\end{lemma}
\begin{proof}
    Let $D^\prime=[a_1,b_1]\times\dots\times[a_d,b_d]$ be the smallest box-domain containing $D$ such that $a_i,b_i\in\Z$ for all $i$. It is enough to prove the lemma for $D^\prime$, since $\dist(v,D^\prime)\leq\dist(v,D)$ for any $v$ and $\lvert D^\prime\rvert\leq c_d\lvert D\rvert$. For $v\in\Z^d$ we write $C(v)$ for the closest point to $v$ in $D^\prime\cap\Z^d$. The idea of the proof is to partition $\{v\;|\;\dist(v,D^\prime)>R)\}$ according to the value of $C(v)$.

    For $i=0,1,\dots,d$ let $\mathrm{Face}_i$ denote the points $x\in D^\prime\cap \Z^d$ such that $d+i$ of their nearest neighbours (in $\Z^d$) are also contained in $D^\prime$. So for example $\mathrm{Face}_0$ denotes the corners of $D^\prime$ and $\mathrm{Face}_d$ denotes the points of $\Z^d$ in the interior of $D^\prime$. We note that the number of points in $\mathrm{Face}_i$ is at most $c \lvert D^\prime\rvert^{\frac{i}{d}}$, where $c$ depends on $d$ and $a$, by elementary geometric considerations.

    We define $S_x=\{v\;|\;\dist(v,D^\prime)>R, C(v)=x\}$. Observe that $S_x=\emptyset$ whenever $x\in\mathrm{Face}_d$. Moreover when $x\in\mathrm{Face}_i$ and $v\in S_x$, $x-v$ must be orthogonal to each of the $i$ directions in which both neighbours of $x$ are contained in $D^\prime$. In other words $S_x-x$ is contained in a subspace of dimension $d-i$ and hence
    \begin{displaymath}
        \sum_{v\in S_x}\lvert v-x\rvert^{-\gamma}\leq \sum_{y\in\Z^{d-i}\;:\;\lvert y\rvert>R}\lvert y\rvert^{-\gamma}\leq c_\gamma R^{-\gamma+d-i}.
    \end{displaymath}
    We then conclude that
    \begin{align*}
        \sum_{v\in\Z^d\;:\;\dist(v,D^\prime)>R}\dist(v,D^\prime)^{-\gamma}=\sum_{i=0}^{d-1}\sum_{x\in\mathrm{Face}_i}\sum_{v\in S_x}\lvert v-x\rvert^{-\gamma}&\leq \sum_{i=0}^{d-1}c_d \lvert D^\prime\rvert^{i/d} c_\gamma R^{-\gamma+d-i}\\
        &\leq c(R^{-\gamma+d}+\lvert D^\prime\rvert^{\frac{d-1}{d}}R^{-\gamma+1})
    \end{align*}
    as required.
\end{proof}

The next ingredient is a `stabilisation' property that $\Delta_0(D_n)$ converges almost surely as $D_n \to\R^d$, which follows essentially from the fact that, by Lemma~\ref{l:Prob_unstable_decay}, the unstable set ${\mathcal{U}_0 = \cup_n \mathcal{U}_0(D_n)}$ is almost surely finite (at least, as long as $\partial D_n$ does not intersect any cube $B_v$ too many times). This requires only the weaker assumption $\beta > d$.

A sequence of box-domains $D_n \to \R^d$ will be called \textit{well-spaced} if the number of indices $n$ for which $\partial D_n$ intersects $B_v$ is bounded over $v \in \Z^d$.

\begin{lemma}[Stabilisation]\label{l:Stabilisation}
For every well-spaced sequence of box-domains $D_n\to\R^d$ there exists a random variable $\Delta_0(\R^d)$ such that 
\begin{displaymath}
\Delta_0(D_n)\xrightarrow{a.s.}\Delta_0(\R^d)\qquad\text{as }n\to\infty.
\end{displaymath}
\end{lemma}

\begin{proof}
Define the set
\[ \mathcal{U}_0:=  \bigcup_n \mathcal{U}_0(D_n) = \{w\in \Z^d\;|\;(f,p_0)\text{ is unstable on }D_n\cap B_w\text{ for some n} \} . \]
By Lemma~\ref{l:Prob_unstable_decay}, the fact that $D_n$ is well-spaced and the Borel-Cantelli lemma, $\mathcal{U}_0$ is finite almost surely. Fix such a realisation of $f$ and $p_0$, and choose $0 < R_1 < R_2$ such that (i) $B_w \subset \Lambda_{R_1}$ for all $w \in \mathcal{U}_0$, and (ii) every bounded component of $\{f=\ell\}$ which intersects $\Lambda_{R_1}$ is contained in $\Lambda_{R_2}$ (i.e.\ does not intersect $\R^d\backslash\Lambda_{R_2}$).

We claim that if $\Lambda_{R_2} \subseteq D_n$, then
\begin{equation}
    \label{e:stab1}
 \Delta_0(D_n)=\overline{N}_\star(\Lambda_{R_1},f,\ell)-\overline{N}_\star(\Lambda_{R_1},\widetilde{f}_0,\ell) ,
 \end{equation}
where $\overline{N}_\star(\Lambda_{R_1},f,\ell)$ denotes the number of bounded components of $\{f\geq\ell\}$ (or $\{f=\ell\}$ if $\star=\mathrm{LS}$) which intersect $\Lambda_{R_1}$. This concludes the proof of the lemma by taking $n \to \infty$.

To prove \eqref{e:stab1}, observe that
\begin{displaymath}
N_\star(D_n,f,\ell)=N_\star(D_n \backslash\Lambda_{R_1},f,\ell)+\overline{N}_\star(\Lambda_{R_1},f,\ell)
\end{displaymath}
which holds because any component of $\{f\geq\ell\}$ (or $\{f=\ell\}$) contained in $D_n$ is either contained in $D_n \setminus\Lambda_{R_1}$ or intersects $\Lambda_{R_1}$, and by the definition of $R_2$, all components of the second type are contained in $D_n$. This also holds if we replace $f$ by $\widetilde{f}_0$ of course, and so
\begin{displaymath}
\Delta_0(D_n)=\Delta_0(D_n \backslash\Lambda_{R_1})+\overline{N}_\star(\Lambda_{R_1},f,\ell)-\overline{N}_\star(\Lambda_{R_1},\widetilde{f}_0,\ell).
\end{displaymath}
Finally, by Lemma~\ref{l:Topological_stability} (applied to $g = f$ and $p = p_0$) and the fact that  $B_w \subset \Lambda_{R_1}$ for all $w \in  \mathcal{U}_0$, we have $\Delta_0(D\backslash\Lambda_{R_1})=0$. Combining these gives \eqref{e:stab1}.
\end{proof}

Finally we will require an ergodic theorem to prove the convergence in \eqref{e:MCLT3}. The following theorem will be well suited to our purposes, as it allows us to sum over translations in $d$ dimensions.
\begin{theorem}[Multi-variate ergodic theorem {\cite[Theorem~25.12]{kal21}}]\label{t:ergodic}
    Let $\xi$ be a random element in some set $S$ with distribution $\mu$. Let $T_1,\dots,T_d$ be $\mu$-preserving transformations of $S$. Assume that the invariant $\sigma$-algebra of each $T_i$ is trivial and let $F\in L^p(\mu)$ for some $p>1$, then as $n_1,\dots,n_d\to\infty$
    \begin{displaymath}
        \frac{1}{n_1\dots n_d}\sum_{i=1}^d\sum_{k_i\leq n_i}F(T_1^{k_1}\dots T_d^{k_d}\xi)\to\E[F(\xi)]
    \end{displaymath}
    where convergence occurs almost surely and in $L^p$.
\end{theorem}
With these results in hand, we are ready to prove Theorem~\ref{t:clt_gen}:

\begin{proof}[Proof of Theorem~\ref{t:clt_gen}]
We first fix a well-spaced sequence of box-domains $D_n \to \R^d$ and prove the result for this sequence, that is, we prove that there exists $\sigma^2 \ge 0$ (possibly depending on $D_n$) such that, as $n\to\infty$
\begin{equation}
\label{e:clt_dn}
\frac{\Var[N_\star(D_n)]}{\lvert D_n\rvert}\to\sigma^2 \quad \text{and} \quad
\frac{N_\star(D_n) -  \E[N_\star(D_n)]}{\lvert D_n\rvert^{1/2}}\xrightarrow{d} \sigma Z ,
\end{equation}
where $Z$ is a standard normal random variable. At the end of the proof we will argue that $\sigma^2$ does not depend on the sequence $D_n$, and also lift the requirement for $D_n$ to be well-spaced.

Recall from Lemma~\ref{l:RegularInf} that it is sufficient for us to verify conditions \eqref{e:MCLT1}--\eqref{e:MCLT4} for $U_{n,v}=\lvert D_n\rvert^{-1/2}\E[\Delta_v(D_n)\;|\;\mathcal{F}_v]$.

 Consider first \eqref{e:MCLT1}. Fix $\eta>0$ and $\eps\in(0,1-9d/\beta)$. By applying the union bound, Markov's inequality, Jensen's inequality, and the third statement of Lemma~\ref{l:Bounded_moments}, we have
\begin{align}
\nonumber \P\Big(\sup_{v\in\Z^d}\lvert U_{n,v}\rvert>\eta\Big)&\leq \eta^{-(2+\eps)} \sum_{v\in\Z^d}\E\big[\lvert U_{n,v}\rvert^{2+\eps}\big] \\
 \nonumber & \leq\eta^{-(2+\eps)}\lvert D_n\rvert^{-\frac{2+\eps}{2}}\sum_{v\in\Z^d}\E\left[\lvert\Delta_v(D_n)\rvert^{2+\eps}\right]\leq c_3 \eta^{-(2+\eps)}\lvert D_n\rvert^{-\frac{\eps}{2}}.
\end{align}
Since this converges to zero as $n\to\infty$, we have verified~\eqref{e:MCLT1}.

We can use similar estimates for \eqref{e:MCLT2}: replacing the supremum by a sum, using the conditional Jensen inequality and the third statement of Lemma~\ref{l:Bounded_moments},
\begin{align*}
\E\Big[\sup_{v\in\Z^d}U_{n,v}^2\Big]&\leq\frac{1}{\lvert D_n\rvert}\sum_{v\in\Z^d}\E\left[\Delta_v(D_n)^2\right]\leq c_3,
\end{align*}
as required. Furthermore by the conditional Jensen inequality, the fact that $\Delta_v$ is integer valued and the above bound
\begin{displaymath}
    \E\left[\sum_{v\in\Z^d}\lvert U_{n,v}\rvert\right]\leq\lvert D_n\rvert^{-1/2}\sum_{v\in\Z^d}\E\left[\lvert\Delta_v(D_n)\rvert\right]\leq\lvert D_n\rvert^{-1/2} \sum_{v\in\Z^d}\E\left[\Delta_v(D_n)^2\right]<\infty
\end{displaymath}
verifying \eqref{e:MCLT4}.

We turn now to \eqref{e:MCLT3}. Letting $\tau_v$ denote translation by $v$ (for $v\in\Z^d$), we note that the sequence of random variables $\Delta_v(D_n)$ for $n\in\N$ has the same distribution as the sequence $\Delta_0(\tau_{-v}D_n)$. Therefore by Lemma~\ref{l:Stabilisation}, for each $v\in\Z^d$ there exists a random variable $\Delta_v(\R^d)$ such that $\Delta_v(D_n)\xrightarrow{a.s.}\Delta_v(\R^d)$ as $n\to\infty$. For $v\in\Z^d$ and a box-domain $D$ let
\begin{displaymath}
X_v(D)=\E[\Delta_v(D)\;|\;\mathcal{F}_{v}] \quad\text{and}\quad X_v=\E[\Delta_v(\R^d)\;|\;\mathcal{F}_{v}].
\end{displaymath}
The statement we need to prove is that
\begin{equation}\label{e:L1convergence}
\lvert D_n\rvert^{-1}\sum_{v\in\Z^d}X_v^2(D_n)\xrightarrow{L^1}\sigma^2
\end{equation}
as $n\to\infty$. Let $V_n = \{w \in \Z^d : D_n \cap B_w \neq \emptyset\}$ denote the indices of cubes which intersect $D_n$. Our strategy is roughly to show that the sum over $v\notin V_n$ is negligible whilst for $v\in V_n$ we can approximate $X_v^2(D_n)$ by $X_v^2$ and hence apply the ergodic theorem. 

Aiming towards the setting of Theorem~\ref{t:ergodic}, we let $\xi=(W_v,W^\prime_v)_{v\in\Z^d}$ denote the pair of white noise processes used to define $f$ and $(\widetilde{f}_v)_{v\in\Z^d}$ and work with the probability space induced by the distribution of $\xi$. Let $T_1,\dots,T_d$ denote translation by distance $1$ in the positive direction of each of the coordinate axes respectively. These are clearly measure preserving transformations of $\xi=(W_v,W^\prime_v)_{v\in\Z^d}$. Moreover since the $W_v$ and $W^\prime_v$ are independent and identically distributed, the $\sigma$-algebra of invariant events associated to each such transformation is trivial (this follows from a standard argument using Kolmogorov's 01-law). In the following paragraph we adjust our notation to emphasise the dependence on the underlying white noise processes: for example we write $X_v(D,\xi):=X_v(D)$, $X_v(\xi):=X_v$ and similarly for other variables. We claim that for any $v\in\Z^d$
\begin{equation}\label{e:ErgodicClaim}
    X_v(\xi)=X_0(\tau_{-v}\xi)
\end{equation}
where $\tau_{-v}\xi=(\tau_{-v}W,\tau_{-v}W^\prime)$ denotes the translated white noise processes defined by $(\tau_{-v}W)_u=W_{u+v}$ for $u\in\Z^d$ (and similarly for $W^\prime$). To prove this, first note that for any box-domain $D$
\begin{displaymath}
    \Delta_v(D,\xi)=\Delta_0(\tau_{-v}D,\tau_{-v}\xi).
\end{displaymath}
This follows from the definition of $\Delta_v$ in \eqref{e:DeltaDef} since translating $W$ and $W^\prime$ by $-v$ is equivalent to translating $f$ and $\widetilde{f}$ by $-v$ (courtesy of \eqref{e:wnr}). Choosing a sequence of box-domains $D_n\to\R^d$ (and noting that this is equivalent to $\tau_{-v}D_n\to\R^d$ for fixed $v$) we see from Lemma~\ref{l:Stabilisation} that
\begin{displaymath}
    \Delta_v(D_n,\xi)\to\Delta_v(\R^d,\xi)\quad\text{and}\quad\Delta_0(\tau_{-v}D_n,\tau_{-v}\xi)\to\Delta_0(\R^d,\tau_{-v}\xi)\quad\text{almost surely}.
\end{displaymath}
Hence these limits coincide. Finally noting that
\begin{equation}\label{e:filtration_stationary}
\begin{aligned}
        \mathcal{F}_0(\tau_{-v}W):=\sigma((\tau_{-v}W)_u\;|\;u\preceq 0)&=\sigma(W_{u+v}\;|\;u\preceq 0)\\
        &=\sigma(W_{u}\;|\;u\preceq v)=:\mathcal{F}_v(W),
\end{aligned}
\end{equation}
where the penultimate equality relies on our use of the lexicographic ordering, we see that
\begin{displaymath}
    X_v(\xi)=\E[\Delta_v(\R^d,\xi)\;|\;\mathcal{F}_v(W)]=\E[\Delta_0(\R^d,\tau_{-v}\xi)\;|\;\mathcal{F}_0(\tau_{-v}W)]=X_0(\tau_{-v}\xi).
\end{displaymath}
This verifies \eqref{e:ErgodicClaim}. By Fatou's lemma, Lemma~\ref{l:Stabilisation} and the first part of Lemma~\ref{l:Bounded_moments}
\begin{equation}\label{e:Fatou}
    \E[\lvert X_0(\xi)\rvert^{2+\epsilon}]\leq\liminf_{n\to\infty}\E[\lvert X_0(D_n,\xi)\rvert^{2+\epsilon}]\leq c_1
\end{equation}
so $X_0^2$ has a finite $(1+\epsilon/2)$-th moment. Hence, noting that $\tau_v=T_1^{v_1}\dots T_d^{v_d}$ for $v\in\Z^d$, we may apply Theorem~\ref{t:ergodic} to conclude that
\begin{equation}\label{e:ergodic}
\lvert D_n\rvert^{-1}\sum_{v\in V_n}X_v^2=\lvert D_n\rvert^{-1}\sum_{v\in V_n}X_0^2(\tau_{-v}\xi)\xrightarrow{L^1}\E[X_0^2]
\end{equation}
as $n\to\infty$.

It remains to compare the left-hand side of \eqref{e:L1convergence} to $\lvert D_n\rvert^{-1}\sum_{v\in V_n} X_v^2$. As such, for each $n\in\N$, choose box-domains $D^-_n \subset D_n \subset D_n^+$ of the form
\begin{displaymath}
D_n^\pm:=\bigcup_{v\in V_n^\pm}B_v
\end{displaymath}
for index sets $V_n^\pm\subset\Z^d$ such that, as $n \to \infty$,
\[  \zeta_n := \max \{  1 - |D^-_n|/|D_n| , |D^+_n|/|D_n| - 1    \} \to 0   \]
and
\[ \eta_n :=\min\{\mathrm{dist}(V_n^-,\Z^d\backslash V_n),\mathrm{dist}(\Z^d\backslash V_n^+,V_n)\}\sim \lvert D_n\rvert^\frac{1-\lambda}{d} \]
where $\lambda>0$ will be specified below. Roughly speaking this means that the distance between $D_n^-$ (resp.\ $D_n^+$) and $D_n$ goes to infinity, but slowly compared to the order of $\lvert D_n\rvert$.

We now show that the contribution to \eqref{e:L1convergence} from $v$ outside $V_n^+$ is negligible; by the second statement of Lemma~\ref{l:Bounded_moments}, for every $\gamma < \beta/3$ there is a $c_\gamma > 0$ such that,
\begin{equation}
    \label{e:outside}
    \begin{aligned}
    \frac{1}{\lvert D_n\rvert}\E\Big[\sum_{v\notin V_n^+}X_v^2(D_n)\Big] \leq c_\gamma \lvert D_n\rvert^2\eta_n^{-\gamma}(\eta_n^d+\eta_n\lvert D_n\rvert^{\frac{d-1}{d}})\leq c_\gamma \lvert D_n\rvert^{3-\lambda/d-(1-\lambda)\frac{\gamma}{d}}.
    \end{aligned}
\end{equation}
Since $\gamma>3d$, this expression will converge to zero provided we choose $\lambda>0$ sufficiently small.

We next claim that the contributions from $v$ inside $V_n^-$ are well-approximated by their stationary counterparts, i.e.\
\begin{equation}
    \label{e:inside}
\lvert D_n\rvert^{-1}\sum_{v\in V_n^-}(X_v^2(D_n)-X_v^2)\xrightarrow{L^1}0.
\end{equation}
Clearly it is sufficient to show that
   \begin{displaymath}
\lim_{n\to\infty}\sup_{v\in V_n^-}\E\left[\left\lvert X_v^2(D_n)-X_v^2\right\rvert\right]=0.
\end{displaymath}
Suppose that this was not true, so there exists some sequence of points $v_n\in V_n^-$ such that
\begin{equation}\label{e:contradiction}
0<\liminf_{n\to\infty}\E\left[\left\lvert X_{v_n}^2(D_n)-X_{v_n}^2\right\rvert\right]=\liminf_{n\to\infty}\E\left[\left\lvert X_{0}^2(\tau_{-v_n}D_n)-X_{0}^2\right\rvert\right].
\end{equation}
We note that $\tau_{-v_n}D_n$ converges to $\R^d$ since $\eta_n\to\infty$. Hence by Lemma~\ref{l:Stabilisation},
\begin{displaymath}
X_{0}^2(\tau_{-v_n}D_n)\xrightarrow{a.s.}X_{0}^2\quad\text{as }n\to\infty.
\end{displaymath}
Moreover by the first statement of Lemma~\ref{l:Bounded_moments}, $X_{0}^2(\tau_{-v_n}D_n)-X_{0}^2$ is uniformly integrable. However these two facts contradict \eqref{e:contradiction} and so \eqref{e:inside} is proved.

Finally, we note that the contributions to \eqref{e:L1convergence} from $V^+_n \setminus V^-_n$ (or $V_n \setminus V^-_n$ for the stationary counterparts) are also negligible. Specifically, by the first statement of Lemma~\ref{l:Bounded_moments}
\begin{equation}\label{e:negligible1}
\lvert D_n\rvert^{-1}\E\Big[\sum_{v\in V_n^+\backslash V_n^-}X_v^2(D_n)\Big]\leq 2\zeta_n\sup_{v\in\Z^d}\E[X_v^2(D_n)]\leq 2c_1 \zeta_n \to 0,
\end{equation}
and similarly by \eqref{e:Fatou}
\begin{equation}\label{e:negligible2}
\lvert D_n\rvert^{-1}\E\Big[\sum_{v\in V_n\backslash V_n^-}X_v^2\Big] \leq c_1 \zeta_n \to 0.
\end{equation}
Combining \eqref{e:outside}, \eqref{e:inside}, \eqref{e:negligible1} and \eqref{e:negligible2} with \eqref{e:ergodic} allows us to conclude that \eqref{e:L1convergence} holds. This verifies the final condition required, and hence proves \eqref{e:clt_dn}.

It remains to argue that $\sigma^2$ is independent of the choice of the sequence $D_n$, and that we may lift the requirement that $D_n$ be well-spaced. To prove the former, suppose $E_n \to \R^d$ and $F_n \to \R^d$ are two well-spaced sequences such that \eqref{e:clt_dn} holds for distinct $\sigma_E^2$ and $\sigma_F^2$ respectively. Consider an alternating sequence $G_n$, that is, $G_n = E_n$ if $n$ is odd, and $G_n = F_n$ if $n$ is even. Since $G_n \to \R^d$ and $G_n$ is also well-spaced, \eqref{e:clt_dn} holds for a constant $\sigma^2_G$, which is in contradiction with the fact that \eqref{e:clt_dn} holds for $\sigma^2_E$ and $\sigma^2_F$ along subsequences of odd, respectively even, indices.

To lift the requirement that $D_n$ be well-spaced, we fix the (unique) value of $\sigma^2$ established in the previous paragraph. Then let $D_n \to \R^d$ be arbitrary, and suppose for the sake of contradiction that \eqref{e:clt_dn} does not hold for $\sigma^2$. Then by compactness there exists an $s \in \R$ and a constant $p \in [0,1]$ such that $p \neq \P[Z \ge s/\sigma] $ (interpreted as $0$ if $\sigma^2 = 0$), satisfying 
\[ \P \big[ |D_n|^{-1/2} (N_\star(D_n) - \E[N_\star(D_n)] )  \ge s \big] \to p \]
along a subsequence. Since $D_n \to \R^d$, one can extract a further subsequence such that $D_n$ is well-spaced. However since \eqref{e:clt_dn} holds for this subsequence, we have a contradiction.
\end{proof}

\begin{remark}
From the above proof (in particular \eqref{e:ergodic}) it is apparent that the limiting variance stated in Theorems~\ref{t:clt} and~\ref{t:clt_gen} is given by
\begin{equation}\label{e:sigma}
\sigma^2=\E\left[\E[\Delta_0(\R^d)\;|\;\mathcal{F}_0]^2\right]=\E\left[\E\left[\lim_{n\to\infty }N_\star(\Lambda_n,f,\ell)-N_\star(\Lambda_n,\widetilde{f}_0,\ell)\;\middle|\;\mathcal{F}_0\right]^2\right] ,
\end{equation}
where $\Delta_0(\R^d)$ is the random variable defined in Lemma \ref{l:Stabilisation} for the sequence $D_n = \Lambda_n$ (in fact one can take any well-spaced $D_n \to \R^d$ in place of $\Lambda_n$). This expression highlights the importance of the choice of filtration for our proof: clearly \eqref{e:sigma} could not hold simultaneously for arbitrary filtrations satisfying the normalisation \eqref{e:Filtration_normalisation}. Our application of the ergodic theorem relies crucially on the equality in \eqref{e:filtration_stationary}. This property holds only for the standard lexicographic order up to reflection/reordering of the axes.
\end{remark}

\begin{remark}
\label{r:beta}
Recall that our proof requires the correlation decay $\beta > 9d$ in Assumption~\ref{a:clt}. This condition arises from Lemma~\ref{l:Bounded_moments}, where it is combined with a third moment bound for critical points (Theorem \ref{t:tmb}) to control the $(2+\eps)$ moments of $\Delta_v(D)$. If higher order moment bounds for critical points were available, then by adjusting the proof of Lemma~\ref{l:Bounded_moments} we could reduce the decay assumption on $\beta$. Specifically, if we knew that $\E[N_c(1)^k] < \infty$ for some integer $k \ge 4$, then we could replace \eqref{e:HolderBM} by
\begin{displaymath}
\E\left[\lvert\Delta_0(D)\rvert^{2+\eps}\right]\leq\sum_{u,v,w\in V}\prod_{i\in\{u,v,w\}}\E\left[\overline{N}_c\left(B_i\right)^k\right]^{\frac{2+\eps}{3k}}\P(i\in \U_0)^{\frac{k-2-\eps}{3k}},
\end{displaymath}
which would allow us to obtain the CLT for all $\beta>3kd/(k-2)$. Interestingly, even if all such moments were known to be finite, this would still only cover the regime $\beta > 3d$, and not the entire short-range correlated regime $\beta>d$.

After this manuscript was submitted for publication, higher order moment bounds for critical points were proven independently in \cite{gs23,al23}. Specifically it was shown that if $f$ is a stationary $C^{k+1}$ field with a continuous spectral density then $\E[N_c(1)^k]<\infty$. Hence if such a field satisfies Assumption~\ref{a:clt} for $\beta>3kd/(k-2)$ then the CLT holds for the component count.
\end{remark}

\subsection{Positivity of the limiting variance}
We now turn to proving Theorem \ref{t:posvar}. Again we work under Assumption~\ref{a:clt}, and take $\star \in \{\mathrm{ES},\mathrm{LS}\}$ and $\ell \in \R$ as fixed.

Let us briefly describe our strategy. Recall that the limiting variance $\sigma^2 = \sigma^2_\star(\ell)$ is given by \eqref{e:sigma}, whose expression involves conditioning on $\mathcal{F}_0$. The first step is to bound $\sigma$ from below by replacing the conditioning on $\mathcal{F}_0$ with conditioning on a single univariate Gaussian $Z$ corresponding to the mean of the white noise on a large box $D$. Then it is sufficient to show that the variance of the mean component count, as $Z$ varies, is strictly positive. In turn, it suffices to show that the mean component count is not constant when a drift (i.e.\ a change in the mean) is added to $Z$. On the other hand, as long as $\int q > 0$, adding such a drift has the effect of shifting the mean of the \textit{field} inside the large box, with boundary-order corrections. So since the component count density (i.e.\ the function $\mu_\star(\ell)$ in \eqref{e:lln}) is strictly positive and tends to zero as the level tends to infinity, provided the box and drift are chosen large enough, the drift necessarily has a non-zero effect on the mean component count, as required.

We now formalise this strategy. Let us begin with a variant of the stabilisation lemma proven above. Recall that $\widetilde{f}_0$ denotes the field $f$ with the white noise in $B_0$ resampled. For brevity we henceforth drop the level $\ell$ from the notation $N_\star(D,g,\ell)$.

\begin{lemma}
\label{l:stab2}
Let $w \in C^4(\R^d)$ be such that there exists $c > 0$ and $\gamma > 3d/2$ so that, for all $x$,
\begin{equation}
    \label{e:stab2con}
\max_{|\alpha| \le 2} |\partial^\alpha w(x)| \le c (1 + |x|)^{-\gamma}.
\end{equation}
Then there exists a random variable $D(w)$ such that, as $n \to \infty$
\[  N_\star(\Lambda_n, f + w ) -  N_\star(\Lambda_n, \widetilde{f}_0 ) \xrightarrow{a.s.} D(w) .  \]
Moreover,
\[ \E[ D(w) ] = \lim_{n \to \infty} \E[ N(\Lambda_n, f + w) ] - \E[ N(\Lambda_n, f ) ] , \]
and $\E[D(w(\cdot))] = \E[D(w(x + \cdot))]$ for any $x \in \R^d$.
\end{lemma}
\begin{remark}
In particular $w = q \ast \id_{D}$ satisfies \eqref{e:stab2con} for any compact $D$. 
\end{remark}
\begin{proof}
Recall that $p_0 = \widetilde{f}_0 - f$. Define the (random) unstable subset
\[ \U_1 = \{v \in \Z^d : (\widetilde{f}_0,w-p_0) \text{ is unstable on $B_v$} \}. \]
By Lemma~\ref{l:quant_stab}, for every $\eps > 0$ there are $c_\eps,c > 0$ such that 
\[ \P\left(v\in \U_1\right)  \le  c_\eps \inf_{\tau > \|w\|_{C^1(v + \Lambda_2)} + c \sqrt{M_v}} \Big( \tau^{1-\eps} +  e^{-(\tau- c \sqrt{M_v} - \|w\|_{C^1(v + \Lambda_2)} )^2 /  (2 M_v ) } \Big),\]
where, as shown in the proof of Lemma~\ref{l:Prob_unstable_decay}, $M_v\leq c_1(1+\lvert v\rvert)^{-2\beta}$. For some $\eta>0$, we now choose
\begin{displaymath}
\tau= c_2(1 + \lvert v\rvert)^{-\min\{\beta,\gamma\}+\eta} \geq 2 ( c\sqrt{M_v} + \|w\|_{C^1(v + \Lambda_2)})
\end{displaymath}
where the inequality holds for an appropriate choice of constant $c_2>0$. We then have
\begin{displaymath}
 \P(v\in \U_1) \le c_3 \Big( ( 1 + \lvert  v\rvert)^{-(\min\{\beta,\gamma\}-\eta)(1-\eps)}+e^{-c_4( 1 + \lvert  v\rvert)^{2\eta}} \Big)
\end{displaymath}
for some $c_3, c_4 > 0$ depending on $\epsilon$. Choosing $\eta,\eps>0$ small enough we can ensure that this expression is bounded by $c_\delta (1 + \lvert  v\rvert)^{-\min\{\beta,\gamma\}+\delta}$ for any $\delta>0$. Since this expression is summable over $v \in \Z^d$, arguing as in the proof of Lemma \ref{l:Stabilisation}, we see that
\[ N_\star(\Lambda_n, f + w )  -  N_\star(\Lambda_n, \widetilde{f}_0 )\]
converges almost surely as $n\to\infty$, proving the first statement.

Turning to the second statement, by Lemma~\ref{l:stab3} (applied to $g = \widetilde{f}_0$ and $p = w-p_0$)
\[ |N_\star(\Lambda_n, f + w ) - N_\star(\Lambda_n, \widetilde{f}_0 )  | \le \sum_{v \in \Z^d} ( N_c(B_v,f+w) + N_c(B_v,\widetilde{f}_0)  ) \id_{ v \in \U_1}. \]
Then by H\"older's inequality, Theorem~\ref{t:tmb}, and the fact that $\min\{\beta,\gamma\}>3d/2$,
\begin{align*}
& \E\bigg[\sum_{v \in \Z^d} ( N_c(B_v,f+w) + N_c(B_v,\widetilde{f}_0)  ) \id_{ v \in \U_1 }\bigg] \\
& \qquad \qquad \le \sum_{v \in \Lambda_n} \E[ ( N_c(B_v,f+w) + N_c(B_v,\widetilde{f}_0)  )^3]^{1/3}\P(v \in \U_1)^{2/3}\\
& \qquad \qquad \le c\sum_{v\in \Z^d}(1+\lvert v\rvert)^{-\frac{2}{3}(\min\{\beta,\gamma\}+\delta)}<\infty
\end{align*}
for $c  > 0$ depending only on $f$ and $\|w\|_{C^4(\R^d)}$, where we have taken $\delta>0$ sufficiently small to ensure that the sum is finite. Thus $| N_\star(\Lambda_n, f + w , \ell) - N_\star(\Lambda_n, \widetilde{f}_0 , \ell)|$ is dominated by a quantity with finite expectation, so by the dominated convergence theorem and equality in law of $\widetilde{f}_0$ and $f$,
\[ \E[ D(w) ] =  \lim_{n \to \infty} \E[ N(\Lambda_n, f + w) -  N(\Lambda_n, \widetilde{f}_0 ) ] =  \lim_{n \to \infty} \E[ N(\Lambda_n, f + w) ] - \E[ N(\Lambda_n, f ) ] , \]
as required. The final claim follows by stationarity, since the argument for the existence of $D(w)$ in the proof of Lemma~\ref{l:Stabilisation} shows that $D(w)$ is unchanged if $\Lambda_n$ is replaced by $\Lambda_n - x$.
\end{proof}

Let us next give a sufficient condition for $\sigma > 0$. Later we will verify (a rescaled version of) this sufficient condition under the additional assumption that $\int q > 0$.

\begin{lemma}
\label{l:sufcon}
Suppose there exists a set $I \subset \R$ of positive measure such that, for all $s \in I$,
\begin{equation}
    \label{e:cond}
 \lim_{n \to \infty} \E[ N_\star(\Lambda_n, f + s (q \ast \id_{B_0} ) , \ell) ] - \E[ N_\star(\Lambda_n, f , \ell) ] < 0 .
 \end{equation}
Then $\sigma  > 0$.
\end{lemma}

\begin{proof}
Consider an orthogonal decomposition of the white noise $W|_{B_0}$ into the Gaussian function $Z_0 \id(\cdot)|_{B_0}$ and an orthogonal part, where $Z_0$ is a standard normal random variable, and observe that $Z_0$ is measurable with respect to $\mathcal{F}_0$. Define the function
\[ F(z) = \E\big[ \Delta_0(\R^d) \; \big| \; Z_0 = z \big] .\]
Then by Jensen's inequality
\[ \sigma^2 = \E\left[\E[\Delta_0(\R^d)\;|\;\mathcal{F}_0]^2\right] \ge \Var[F(Z_0)] . \]
On the other hand, for every $s \in \R$, by definition
\[ \E[F(Z_0 + s)] = \E[D(s(q \ast \id_{B_0})] . \]
Hence, by Lemma \ref{l:stab2} and \eqref{e:cond}, $\E[F(Z_0+s)] < 0$ on a set $s \in I$ of positive measure, which implies that $\Var[F(Z_0)] > 0$ as required.
\end{proof}

Towards verifying \eqref{e:cond}, let us collect some facts about the component count density $\mu = \mu_\star(\ell)$ in \eqref{e:lln}:
\begin{lemma}\label{l:MeanFunctional}
The following hold:
\begin{enumerate}
    \item For all $\ell\in\R$, $\mu_\star(\ell)>0$.
    \item As $\ell \to \infty$, $\mu_\star(\ell)\to 0$.
    \item For all $\ell\in\R$, $\E[N_\star(\Lambda_n,f,\ell)] / n^d \to \mu_\star(\ell)$ as $n\to\infty$.
\end{enumerate}
\end{lemma}
\begin{proof}
\textbf{(1).} By Lemma \ref{l:wnr}, the support of the spectral measure of $f$ contains an open set. By \cite[Appendix~C.2 and Theorem~1]{ns16}, this ensures that $\mu_\mathrm{LS}(0)>0$, and the proof of this result applies equally well to all $\ell\in\R$ for both excursion and level sets.

\textbf{(2).} We first claim that
\begin{equation}
\label{e:highell}
\mu_\star(\ell)\leq\E[N_c(\Lambda_1,f,[\ell,\infty))],
\end{equation}
where $N_c(\Lambda_1,f,[\ell,\infty))$ denotes the number of critical points of $f$ in $\Lambda_1$ with level at least~$\ell$. To see this, observe that $N_\star(\Lambda_n,f,\ell)$ is bounded by the number of stratified critical points of $f$ in $\Lambda_n$ with height at least $\ell$. This is a standard consequence of Morse theory: when raising the level, the topology of the excursion/level set can only change upon passing through a stratified critical point and for each such point the component count can change by at most one. Moreover when the level is sufficiently high, the excursion/level set is empty so we must have passed through at least one critical point for each component. (See \cite{gp88} for a comprehensive background on the Morse theory of stratified spaces or \cite{han02} for a concise introduction.) Since the expected number of stratified critical points in a boundary stratum of $\Lambda_n$ scales like $n^{k}$ where $k\in\{0,1,\dots,d-1\}$ is the dimension of the stratum, by the definition of $\mu_\star$ and stationarity of critical points we have \eqref{e:highell}. To finish the proof we claim that
\[ \E[N_c(\Lambda_1,f,[\ell,\infty))] \to 0 \]
as $\ell \to \infty$. Indeed $ N_c(\Lambda_1,f,[\ell,\infty) ) \to 0$ almost surely, and ${N_c(\Lambda_1,f,[\ell,\infty) ) \le N_c(\Lambda_1, f)}$ which has finite expectation. Then the claim follows from the dominated convergence theorem.

\textbf{(3).} This follows immediately from the $L^1$ convergence in \eqref{e:lln}.
\end{proof}

We can now finish the proof of Theorem \ref{t:posvar}:

\begin{proof}[Proof of Theorem~\ref{t:posvar}]
Observe that one can run the proof of the CLT by decomposing the white noise $W$ over any lattice $m \Z^d$, $m \in \N$, in place of $\Z^d$, and the limiting variance must be the same. Hence by Lemma~\ref{l:sufcon} it is enough to find a sufficiently large $m \in N$ and a set $I$ of positive measure such that, for all $s \in I$,
\[ \lim_{n \to \infty} \E[ N_\star(\Lambda_n , f + s (q \ast (2m B_0)  ) ) ] - \E[ N_\star(\Lambda_n, f ) ] < 0 .\]
Since $\int_{\R^d}q(x)\;dx>0$ by assumption, it is sufficient for us to verify the above condition with $\overline{q}:=q/\int_{\R^d}q(x)dx$ replacing $q$. Moreover, by stationarity we may replace $2m B_0$ with $\Lambda_m$.

By the first two points of Lemma~\ref{l:MeanFunctional} there exists an open interval $I$ and $\eta>0$ such that $\mu_\star(\ell-s)-\mu_\star(\ell)<-7\eta$ for all $s\in I$. Henceforth we fix such an $s\in I$. By the third point of Lemma~\ref{l:MeanFunctional}, as $k \to \infty$,
\begin{displaymath}
    \E[N_\star(\Lambda_k,f+s)] / k^d \to \mu_\star(\ell-s)\quad\text{and}\quad\E[N_\star(\Lambda_k,f)] / k^d \to \mu_\star(\ell).
\end{displaymath}
Combining these observations, for $k$ sufficiently large, 
\begin{equation}\label{e:Pos1}
\E[N_\star(\Lambda_k,f+s)]-\E[N_\star(\Lambda_k,f)]<-6\eta k^d .
\end{equation}

We now state three claims which will be combined to prove the theorem. We let $w:=w_m:=s\overline{q}\star\ind_{\Lambda_m}$, and observe that $\|w\|_{C^4(\R^d)} \le \lvert s \rvert \max_{|\alpha| \le 4} \|\partial^\alpha \overline{q}\|_{L^1(\R^d)}$ is uniformly bounded over $m$. Set $k=m-\lceil\sqrt{m}\rceil$ (although from the proof it will be apparent that we could choose $k=m-r$ for any $1\ll r\ll m$). We claim that for some $m$ sufficiently large and all $n>m$,
\begin{align}
\label{e:Pos2}    \E[\lvert N_\star(\Lambda_n,g)-N_\star(\Lambda_k,g)-N_\star(\Lambda_n\setminus\Lambda_k,g)\rvert]&\leq\eta k^{d-1} \le \eta m^d \quad\text{for }g=f,f+w,\\
\label{e:Pos3}    \E[\lvert N_\star(\Lambda_k,f+w)-N_\star(\Lambda_k,f+s)\rvert]&\leq\eta m^d,\\
\label{e:Pos4}    \E[\lvert N_\star(\Lambda_n\setminus\Lambda_k,f+w)-N_\star(\Lambda_n\setminus\Lambda_k,f)\rvert]&\leq\eta m^d.
\end{align}

To explain the intuition behind \eqref{e:Pos2}--\eqref{e:Pos4}, recall that components inside $\Lambda_n$ are either  inside $\Lambda_k$, inside $\Lambda_n \setminus \Lambda_k$, or hit the boundary of $\Lambda_k$. Roughly speaking, \eqref{e:Pos2} means that the expected number of components of the third type is negligible (compared to the volume of $\Lambda_m$), \eqref{e:Pos3} means that $w$ is almost a constant on $\Lambda_k$, so the expected number of components of the first type is indifferent as to whether we perturb by $w$ or by a constant, and \eqref{e:Pos4} means that perturbation by $w$ has negligible effect on components of the second type.

Combining these equations with \eqref{e:Pos1}, the triangle inequality, and the fact that $k/m\to 1$ immediately gives for $m$ sufficiently large, and all $n>m$,
\begin{displaymath}
\E[ N_\star(\Lambda_n , f + s (\bar q \ast \id_{\Lambda_m} ) ) ] - \E[ N_\star(\Lambda_n, f ) ] < -\eta m^d,
\end{displaymath}
completing the proof of the theorem.

It remains to verify claims \eqref{e:Pos2}--\eqref{e:Pos4}. In the sequel $c' > 0$ will denote a constant that depends only on $f$ and $s$ and may change from line to line. Observe that each component of $g = f, f+w$ inside $\Lambda_n$ must either be contained in $\Lambda_k$, be contained in $\Lambda_n\setminus\Lambda_k$ or intersect $\partial\Lambda_k$. The number of components intersecting $\partial\Lambda_k$ is dominated by the number of critical points of $g$ restricted to the boundary, therefore
\begin{displaymath}
\lvert N_\star(\Lambda_n,g)-N_\star(\Lambda_k,g)-N_\star(\Lambda_n\setminus\Lambda_k,g)\rvert\leq N_c(\partial\Lambda_k,g).
\end{displaymath}
 By Jensen's inequality applied to Theorem~\ref{t:tmb}, the expectation of the right hand side here is at most $c'k^{d-1}$, verifying \eqref{e:Pos2}.

As in the proof of Lemma \ref{l:stab2}, by Lemma~\ref{l:stab3}, H\"older's inequality and Theorem~\ref{t:tmb},
\begin{align}
\nonumber \E[\lvert N_\star(\Lambda_k,f+w)-&N_\star(\Lambda_k,f+s)\rvert]\\
\nonumber&\leq\sum_{v\in\Z^d\cap\Lambda_k}\E[(N_{c}(B_v,f+s) + N_c(B_v,f+w))^3]^\frac{1}{3}\P(v\in \U_2)^\frac{2}{3}\\
   \label{e:Pos5}  &\leq c'\sum_{v\in\Z^d\cap\Lambda_k}\P(v\in \U_2)^\frac{2}{3}\leq c'm^d\sup_{v\in\Z^d\cap\Lambda_k}\P(v\in \U_2)^\frac{2}{3},
\end{align}
where 
\begin{displaymath}
\U_2:=\{v \in \Z^d : (f, w-s) \text{ is unstable on $B_v$ at level }\ell-s\}.
\end{displaymath}
We claim that
\begin{displaymath}
\int_{\R^d}\overline{q}(u)\;du=1 \quad\text{and}\quad\int_{\R^d}\partial^\alpha \overline{q}(u)\;du=0\quad\text{for }\lvert\alpha\rvert=1.
\end{displaymath}
The first property just follows from $\overline{q}$ being defined as the normalisation of $q$ whilst the second follows from the decay of $q$. For $x\in\Lambda_k$ and $\lvert\alpha\rvert\leq 1$ we then have
\begin{align*}
\lvert \partial^\alpha(w(x)-s)\rvert=\lvert s\rvert\bigg\lvert \int_{\Lambda_m}\partial^\alpha\overline{q}(x-u)du-\int_{\R^d}\partial^\alpha\overline{q}(x-u)\;du\bigg\rvert&\leq \lvert s\rvert\int_{\R^d\setminus\Lambda_{\sqrt{m}}}\lvert\partial^\alpha\overline{q}(u)\rvert du\\
&\leq c'm^{-(\beta-d)/2}.
\end{align*}
Hence $\|w-s\|_{C^1(v + \Lambda_2)}\leq c' m^{-(\beta-d)/2}$ for $v\in\Lambda_k$, and so by Lemma~\ref{l:quant_stab}, for every $\epsilon>0$ we have $\P(v\in \U_2)\leq c_\epsilon m^{-(1-\epsilon)(\beta-d)/2}$. Since $\beta>d$, combining this with \eqref{e:Pos5} proves \eqref{e:Pos3}.

Repeating the arguments to justify \eqref{e:Pos5} shows that
\begin{equation}\label{e:Pos6}
\begin{aligned}
\E[\lvert N_\star(\Lambda_n\setminus\Lambda_k,f+w)-N_\star(\Lambda_n\setminus\Lambda_k,f)\rvert]&\leq c' \sum_{v\in\Z^d\cap(\Lambda_n\setminus\Lambda_k)}\P(v\in \U_3)^\frac{2}{3}
\end{aligned}
\end{equation}
where 
\begin{displaymath}
\U_3:=\{v \in \Z^d : (f, w) \text{ is unstable on $B_v$ at level }\ell\}.
\end{displaymath}
For $x\in\R^d$ and a multi-index $\alpha$ such that $\lvert\alpha\rvert\leq 1$
\begin{align*}
    \lvert \partial^\alpha w(x)\rvert&=\Big\lvert  \int_{\Lambda_m}\partial^\alpha\overline{q}(x-u)\;du\Big\rvert\leq \int_{\lvert u\rvert>\mathrm{dist}(x,\Lambda_m)}\lvert\partial^\alpha\overline{q}(u)\rvert\;du\leq c' (1+\dist(x,\Lambda_m))^{-\beta+d}.
\end{align*}
Therefore by Lemma~\ref{l:quant_stab}
\begin{displaymath}
\P(v\in  \U_3)\leq c_\epsilon(1+\mathrm{dist}(v,\Lambda_m))^{-(1-\epsilon)(\beta-d)}
\end{displaymath}
for any $\epsilon>0$. Hence by Lemma~\ref{l:ElementarySum}, \eqref{e:Pos6} is bounded above by
\begin{align*}
    \sum_{v\in\Z^d\setminus\Lambda_{m-\sqrt{m}}}\P(v\in \U_3)^{2/3}&\leq c' m^{d-\frac{1}{2}}+\sum_{v\in\Z^d\setminus\Lambda_{m+\sqrt{m}}}c_\epsilon(1+\mathrm{dist}(v,\Lambda_m))^{-\frac{2(1-\epsilon)}{3}(\beta-d)}\\
    &\leq c' m^{d-\frac{1}{2}}+c\big(m^{-\frac{(1-\epsilon)}{3}(\beta-d)+d/2}+m^{d-1-\frac{(1-\epsilon)}{3}(\beta-d)+1/2}\big)\\
    &\leq c' m^{d-\frac{1}{2}}
\end{align*}
for $\epsilon>0$ sufficiently small, since $\beta>5d/2$. Combined with \eqref{e:Pos6} this verifies \eqref{e:Pos4}, and thus completes the proof of the theorem.
\end{proof}

\begin{remark}
It is interesting to note that the proof of Theorem \ref{t:posvar} only requires that Assumption \ref{a:clt} holds for $\beta > 5d/2$ (taking \eqref{e:sigma} as the definition of $\sigma$). Moreover if one knew that the conclusion of Theorem~\ref{t:tmb} (including the uniformity over $\|p\|\leq\tau$) held for all higher order moments then the above proof would require only $\beta>d$. This suggests that a similar strategy could show that $\Var[N_\star(R,\ell)]$ is of at least volume order for all $\beta > d$.
\end{remark}

\medskip

\section{Analysis of critical points}
\label{s:dbcp}

In this section we prove the third moment bound on critical points in Theorem \ref{t:tmb}. The proof exploits the `divided difference' method, which originated in \cite{bel66, cuz75} for studying zeros of one-dimensional Gaussian processes and was used recently in \cite{al21} to prove quite comprehensive results in this setting. We introduce this technique in Section \ref{s:introdd}. In principle our analysis could be extended to moments of higher order, although this would require more technical arguments (and stronger assumptions on the field).

\subsection{Reduction to a bound on the three-point intensity of critical points}
We first reduce the proof of Theorem \ref{t:tmb} to a pointwise bound on critical point intensities.

Henceforth we assume that $f$ satisfies Assumption~\ref{a:gen} and $p \in C^4(\R^d)$. We abbreviate $F=f+p$, and recall that $N_c(R)=N_c([-R,R]^d)$ denotes the number of critical points of $F$ in $\Lambda_R = [-R,R]^d$. By the Kac-Rice theorem (see \cite[Chapter~11]{at07} or \cite[Chapter~6]{aw09} for background), for every $R > 0$
\begin{equation}
    \label{e:krtp}
  \mathbb{E}[N_c(R)(N_c(R)-1)(N_c(R)-2)]=\int_{x,y,z\in \Lambda_R} J(x,y,z) \, dxdydz ,  
  \end{equation} 
where $J(x,y,z)$ is the three-point intensity function
\[ J(x,y,z) = \varphi(x,y,z) \times \tilde{\E} \big[  \big| \det ( \nabla^2 F(x)  \nabla^2 F(y) \nabla^2 F(z) ) \big| \big] \]
and $\widetilde{\E}$ denotes the expectation under the conditioning
\begin{equation}
    \label{e:cond5}
\big(\nabla F(x), \nabla F(y), \nabla F(z) \big) = (0,0,0) , 
\end{equation}
with $\varphi(x,y,z)$ the corresponding Gaussian density. This formula is valid since $f$ is $C^2$-smooth and $(\nabla f(x), \nabla f(y), \nabla f(z))$ is non-degenerate for distinct $x,y,z$ (see Lemma \ref{l:nondegen}).

A priori $J(x,y,z)$ may diverge on the diagonal (i.e.\ as $x,y \to z$ or $x-y \to 0$ etc.), but the following bound provides sufficient integrability. We define
\begin{displaymath}
\D=\{(x,y)\in\R^{2d}\;|\;0<\lvert x\rvert<\lvert y\rvert<\lvert x-y\rvert\leq 1\}.
\end{displaymath}

\begin{proposition}
\label{p:tpint}
There exists $c > 0$, depending only on $f$, such that, for all $(x,y)\in\D$,
\[  J(x,y,0)  \le c(1 + \|p\|_{C^4(\R^d)})^{3d} |x|^{-d+1} |y|^{-d+1} (|y|+\sin\theta)^{-d+1}(\sqrt{|y|}+\sin\theta) , \]
where $\theta\in[\pi/3,\pi]$ denotes the angle between $x$ and $y$.
\end{proposition}
\begin{remark}
Throughout this section, in the case that $d=1$ we take $\theta=\pi$ identically. All of our later calculations will be valid with this convention (although many formulae simplify considerably in this case).
\end{remark}

\begin{proof}[Proof of Theorem~\ref{t:tmb} given Proposition \ref{p:tpint}]
By H\"older's inequality, for every $R \ge 1$ and $\eps \in (0,1]$ we have
\begin{align}
\nonumber \E[N_c(R)^3]&\leq \E \Big[\Big(\sum_{x\in 2\eps\Z^d\cap \Lambda_{2R}} N_c(x+\Lambda_\eps)\Big)^3\Big]\\
\nonumber &=\sum_{x,y,z\in 2\eps\Z^d\cap \Lambda_{2R}}\E\Big[N_c(x+\Lambda_\eps)N_c(y+\Lambda_\eps)N_c(z+\Lambda_\eps)\Big]\\
\label{e:tmb1} &\leq (3R\eps^{-1})^{3d}\sup_{a\in\R^d}\E\big[N_c(a+\Lambda_\epsilon)^3\big].
\end{align}
We henceforth fix $\eps = 1/\sqrt{d}$.

Next we observe that, since $f$ is stationary,
\begin{displaymath}
J(x,y,z)=\widetilde{J}(x-z,y-z,0)
\end{displaymath}
where $\widetilde{J}$ is defined analogously to $J$ when $p(\cdot)$ is replaced by $p(z+\cdot)$. Using the fact that $\|p(z+\cdot)\|_{C^4(\R^d)}=\|p\|_{C^4(\R^d)}$, we deduce from Proposition~\ref{p:tpint} that for all $(x-z,y-z)\in\D$,
\[  J(x,y,z)  \le c_1  |x-z|^{-d+1} |y-z|^{-d+1} (|y-z|+\sin\theta)^{-d+1}(\sqrt{\lvert y-z\rvert}+\sin\theta) , \]
where $\theta\in[\pi/3,\pi]$ denotes the angle between $x-z$ and $y-z$, and $c_1 > 0$ depends only on $f$ and $p$.

Then by the Kac-Rice formula \eqref{e:krtp},
\begin{align}
\nonumber \mathbb{E}[N_c(\eps)(N_c(\eps)-1)&(N_c(\eps)-2)]  =\int_{x,y,z\in \Lambda_\eps} J(x,y,z) \, dxdydz \\
\nonumber & \le c_2 \int_{\lvert x-z\rvert < \lvert y-z\rvert<\lvert x-y\rvert \leq 1} J(x,y,z) \, dxdydz \\
\label{e:kr1} & \le c_3 \int_{\lvert u\rvert<\lvert v\rvert\leq 1}\lvert u\rvert^{-d+1}\lvert v\rvert^{-d+1}(\lvert v\rvert+\sin\theta)^{-d+1}(\sqrt{\lvert v\rvert}+\sin\theta) \, du dv
\end{align}
where $\theta\in[0,\pi]$ is the angle between $u$ and $v$. In the second line we integrated over configurations where $(x,y)$ is the longest side of the triangle with vertices $x$, $y$ and $z$. Other integrals of this type are obtained by relabeling variables and are of the same order. When $d=1$ we observe that the integrand here is bounded, completing the proof in this case.

Assuming now that $d\geq 2$, we switch to spherical coordinates for $u$; for each fixed value of $v$ we choose a basis for $u$ in which the final coordinate is in the direction $v$ and we then define $r>0$, $\theta_1\in[0,2\pi)$ and $\theta_2,\dots,\theta_{d-1}\in[0,\pi]$ by
\begin{align*}
\begin{cases}
    u_1  &  \! \! \! \! \! =r\prod_{k=1}^{d-1}\sin\theta_k\\
    u_m &  \! \! \! \! \! =r\cos\theta_{m-1}\prod_{k=m}^{d-1}\sin\theta_k\qquad\text{for }m=2,\dots,d-1,\\
    u_d &  \! \! \! \! \! =r\cos\theta_{d-1}.
\end{cases}
\end{align*}
Our choice of basis implies that $\theta_{d-1}$ is equal to $\theta$, the angle between $u$ and $v$. The Jacobian for spherical coordinates is given by
\begin{displaymath}
(-r)^{d-1}\prod_{k=2}^{d-1}\sin^{k-1}\theta_k
\end{displaymath}
and so the integral in \eqref{e:kr1} is bounded by
\begin{align*}
c_4 &\int_{\lvert v\rvert\leq 1}\lvert v\rvert^{-d+1}\int_0^{\lvert v\rvert}\int_0^\pi \left(\frac{\sin\theta}{\lvert v\rvert+\sin\theta}\right)^{d-2}\frac{\sqrt{\lvert v\rvert}+\sin\theta}{\lvert v\rvert+\sin\theta}\;d\theta drdv\leq c_5  \int_{\lvert v\rvert\leq 1}\lvert v\rvert^{-d+3/2}\;dv
\end{align*}
where we have used the bound $\frac{\sqrt{\lvert v\rvert}+\sin\theta}{\lvert v\rvert+\sin\theta}\leq\sqrt{\lvert v\rvert}/\lvert v\rvert$. Since the above expression is finite, we see that $\mathbb{E}[N_c(\eps)(N_c(\eps)-1)(N_c(\eps)-2)]$ is bounded by a constant depending only on the distribution of $f$ and the norm of $p$. Moreover the same is true if we replace $\Lambda_\epsilon$ by $a+\Lambda_\epsilon$ since the distribution of $f$ and the norm of $p$ are translation invariant. Since, for any positive random variable $X$,
\[ \E[X^3] = \E[X^3 \id_{X \le 6}] + \E[X^3 \id_{X \ge 6}]  \le 6^3 + 2\E[X(X-1)(X-2)] , \]
combining with \eqref{e:tmb1} we have the desired result.
\end{proof}

\begin{remark}
\label{r:const}
This proof shows that the constant $c$ in Theorem \ref{t:tmb} is at most ${c' (1 +  \|p\|_{C^4(\R^d)})^{3d}}$, for $c' > 0$ depending only on $f$.
\end{remark}

\subsection{Overview of the divided difference method; a second moment bound for zeros}
\label{s:introdd}
We prove Proposition \ref{p:tpint} using the divided difference method, which we present here in a simpler setting. Suppose we want to bound the second moment of the number of zeros of a Gaussian process $F = f + p$, where $f$ is a $C^2$-smooth stationary Gaussian process and $p \in C^2(\R)$. We will show that, as long as the spectral measure of $f$ has more than two points in its support (to ensure sufficient non-degeneracy),
\begin{equation}
    \label{e:zbound}
\E[N(R)^2] \le c (1 +  \|p\|_{C^2(\R)})^2 R^2 , 
\end{equation} 
where $N(R) = |\{x \in [0,R] : F(x) = 0\}|$ is the number of zeros of $F$ in $[0,R]$, and $ c > 0$ depends only on the law of  $f$. Although the bound $\E[N(R)^2] \le c R^2$ is classical (see \cite{kl06}), we are not aware of any existing results on the dependence of the constant on $p$.

Using the Kac-Rice formula as above (valid since $(f(x),f(y))$ is non-degenerate for ${x \neq y}$), it is sufficient to prove that, for $x,y \in [0, 1]$,
\begin{equation}
    \label{e:secterm3}
 J(x,y) \le c (1 +  \|p\|_{C^2(\R)} )^2,
 \end{equation}
where
\[ J(x,y) = \varphi(x,y) \times \mathbb{E}  \big[ | F^\prime(x)   F^\prime(y)  | \, | \,  ( F(x) ,  F(y) ) = ( 0, 0) \big] \]
and $\varphi(x,y)$ is the Gaussian density of $(F(x), F(y) )$ at $(0,0)$. Note that the intensity $J(x,y)$ factorises into a Gaussian density $\varphi$ which diverges as $x - y \to 0$, and a contribution from the gradients which one expects to be small as $x - y \to 0$. To ensure integrability, one needs to carefully control these terms on the diagonal.

The essence of the divided difference method is the observation that, to make the behaviour on the diagonal more amenable, it is helpful to replace the vector $(F(x), F(y))$ with the vector $(F(x), D_{x,y}(F))$, where  \[  D_{x,y}(F) = \begin{cases}  \frac{F(y)-F(x)}{y-x} & x \neq y, \\ F'(x) & x = y . \end{cases} \]

We first show that 
\begin{equation}
    \label{e:secterm2}
 \varphi(x,y) \leq c |y-x|^{-1}
 \end{equation}
for all $x,y \in [0,1]$ and a constant $c > 0$ depending only on $f$.

 For a square-integrable random vector $X$, let $\DC(X)$ denote the determinant of its covariance matrix. Clearly,
\[  \varphi(x,y) \leq (2\pi)^{-1}  \DC( f(x), f(y))^{-1/2}  .\]
Moreover, by standard properties of the DC operator (see Lemma \ref{l:dc})
\[ \DC( f(x), f(y) ) = |x-y|^2 \DC( f(x), D_{x,y}(f)) .\]
Since $f$ is $C^1$-smooth, as $y \to x$,
\[  D_{x,y}(f) = \frac{f(y)-f(x)}{y-x} \to f'(x)  \]
almost surely (and so in $L^2$), and hence we infer that (again see Lemma \ref{l:dc}), as $y \to x$,
\[ \DC( f(x), D_{x,y}(f)) \to  \DC( f(x), f'(x) )  .\]
Since $(f(x), f'(x))$ is non-degenerate (in fact they are independent by stationarity), by continuity and compactness we have \eqref{e:secterm2}.

We next show that 
\begin{equation}
\label{e:secterm4}
 \E[ |  F'(x)   F'(y)  | \, | \, (F(x) , F(y)   ) = (0,0) ]  \le  c ( 1 +  \|p\|_{C^2(\R)} )^2 |y-x|
\end{equation} 
for all $x,y \in [0,1]$. Combined with \eqref{e:secterm2}, this completes the proof of \eqref{e:secterm3} and hence \eqref{e:zbound}. 

By a Taylor expansion, for $z = x,y$,
\[ \big| F'(z) - D_{x,y}(F)  \big| \le |y-x|^{1/2} \|F\|_{C^{3/2}([0,1])} . \]
Hence using standard properties of Gaussian conditioning (see Lemma \ref{l:Gaussian_conditioning}), 
\begin{align*}
     \E[ | F'(z)  |^2  \, |  \, ( F(x) ,F(y)  ) = (0,0) ] & =  \E[ | F'(z)  |^2  \, |  \, ( F(x) ,D_{x,y}(F)  ) = (0,0) ]  \\ 
     & =  \E[ | F'(z) - D_{x,y}(F)  |^2  \, |  \, ( F(x) ,D_{x,y}(F)  ) = (0,0) ]  \\ 
     & \le  |y-x| \cdot \E \big[  \| F \|^2_{C^{3/2}([0,1])} \, | \, (F(x) , D_{x,y}(F)   ) = (0,0)   \big] . 
     \end{align*}
Applying the Cauchy-Schwarz inequality to \eqref{e:secterm4}, it remains to prove that
 \[ \E \big[  \| F \|^2_{C^{3/2}([0,1])} \, | \, (F(x) , D_{x,y}(F)   ) = (0,0)   \big]  \le c (1 + \|p\|_{C^2(\R)})^2 . \]
For this it is sufficient that (see the proof of Lemma \ref{l:C4norm} below for details)
\begin{equation}
\label{e:secterm5}
\sup_{z,x,y\in [0,1]} \sup_{\lvert\alpha\rvert\leq 2} \Var[ \partial^\alpha F(z) \, | \, (F(x) , D_{x,y}(F)   ) = (0,0) ] \le c_1    
\end{equation}
and
\begin{equation}
\label{e:secterm6}
\sup_{z,x,y\in [0,1]}  \sup_{\lvert\alpha\rvert\leq 2}  \lvert\E[ \partial^\alpha F(z) \, | \, (F(x) , D_{x,y}(F)   ) = (0,0)  ] \rvert  \le c_2  \|p\|_{C^2(\R)}. 
\end{equation}
Since conditioning can only reduce the variance of a Gaussian variable, and by stationarity,
\[ \Var[ \partial^\alpha F(z) \, | \, (F(x) , D_{x,y}(F)   ) = (0,0) ] \le  \Var[ \partial^\alpha F(z)  ] =\Var[ \partial^\alpha f (0)] ,\]
which proves \eqref{e:secterm5}. Moreover, using Gaussian regression and the fact that various covariances involving $F$ are the same as the corresponding covariances for $f$,
\begin{align*}
& \E[ \partial^\alpha F (z) \, | \, (F(x) , D_{x,y}(F)   ) = (0,0) ] \\
& \qquad \quad   =   \partial^\alpha p(z) - \Cov\bigg[   \partial^\alpha f (z),\begin{pmatrix}   f(x)\\D_{x,y}(f))   \end{pmatrix}\bigg]^t\Cov[f(x),D_{x,y}(f)]^{-1}(  p(x),D_{x,y}(p)).
\end{align*}
By an argument similar to that following \eqref{e:secterm2}, since $(f(0), f'(0))$ is non-degenerate the spectral norm of $\Cov[f(x),D_{x,y}(f)]^{-1}$ is uniformly bounded. Since also $D_{x,y}(g) \le \|g\|_{C^2([0,1])}$ for $g = f,p$, we deduce \eqref{e:secterm6}.

\subsection{Bound on the three-point intensity}
We now complete the proof of Proposition \ref{p:tpint}, for which we use a natural extension of the above approach to higher order derivatives. We shall prove the following two bounds:

\begin{lemma}\label{l:Third_mom_2}
There exists a $c>0$ such that, for all $(x,y)\in\D$,
\begin{displaymath}
\DC( \nabla f(0), \nabla f(x) , \nabla f(y) ) \ge c \lvert x\rvert^{2d}\lvert y\rvert^{2(d+1)}(\lvert y\rvert+\sin\theta)^{2(d-1)} ,
\end{displaymath}
where $\theta\in[\pi/3,\pi]$ denotes the angle between $x$ and $y$.
\end{lemma}

\begin{lemma}\label{l:Third_mom_1}
There exists a $c > 0$ such that, for all $(x,y)\in\D$,
\begin{multline*}
\E \big[ \big| \det( \nabla^2 F(0) \nabla^2 F(x) \nabla^2 F(y) ) \big|  \: \big| \: (\nabla F(0) ,\nabla F(x) , \nabla F(y) ) = (0,0,0) \big]\\
\leq c(1+\|p\|_{C^4(\R^d)})^{3d} \lvert x\rvert\lvert y\rvert^2(\sqrt{\lvert y\rvert}+\sin\theta) , 
\end{multline*}
where $\theta\in[\pi/3,\pi]$ denotes the angle between $x$ and $y$.
\end{lemma}

\begin{proof}[Proof of Proposition \ref{p:tpint}]
Since 
\begin{displaymath}
\varphi(x,y,0) \le (2 \pi)^{-3d/2}  \DC( \nabla f(0), \nabla f(x) , \nabla f(y) )^{-1/2},
\end{displaymath}
combining Lemmas \ref{l:Third_mom_2} and \ref{l:Third_mom_1} gives the result.
\end{proof}

We now turn to proving Lemmas~\ref{l:Third_mom_2} and \ref{l:Third_mom_1}. To begin, we collect some properties of the $\DC$ operator:
 
\begin{lemma}
\label{l:dc}
Let $X$ and $Y$ be square-integrable random vectors of dimensions $m_1$ and $m_2$ respectively. Then:
\begin{enumerate}
    \item For $A \in \R^{m_1 \times m_1}$, \[ \DC(AX) = \det(A)^2\DC(X).\]
    In particular, for $a\in\R$ and $B\in\R^{m_1\times m_2}$
    \begin{displaymath}
    \DC(aX,Y)=a^{2m_1}\DC(X,Y) \quad \text{and} \quad \DC(X+BY,Y)=\DC(X,Y).
    \end{displaymath}
    \item If $\DC(X,Y)>0$ then $\DC(X)>0$.
    \item If $X_n$ is a sequence converging to $X$ in $L^2$, then $\DC(X_n)\to\DC(X)$.
\end{enumerate}
\end{lemma}
\begin{proof}
The first item is immediate from the definition
\begin{displaymath}
\Cov[X]=\E\big[(X-\E X)(X-\E X)^t\big], 
\end{displaymath}
the second item is elementary, and the third item follows from the continuity of the determinant operator.
\end{proof}

We also collect some standard properties of (Gaussian) conditioning:

\begin{lemma}
\label{l:Gaussian_conditioning}
Let $(X,Y)$ be an $(n+1)$-dimensional non-degenerate Gaussian random vector. Then:
\begin{enumerate}
\item Let $H:\R\to\R$ be continuous and satisfy $\lvert H(t)\rvert\leq C(1+\lvert t\rvert)^m$ for some $C,m>0$ and all $t\in\R$. For $a\in\R^n$
\begin{displaymath}
\E \big[ H(Y) \, \big|\, X=0 \big]= \E\big[H(Y-a\cdot X) \, \big|\, X = 0 \big] .
\end{displaymath}
\item If $A\in\R^{n\times n}$ is invertible, then for every $x\in\R^{n}$,
    \begin{displaymath}
    \E[Y\;|\;X=x]=\E[Y\;|\;AX=Ax].
    \end{displaymath}
\end{enumerate}
\end{lemma}
These properties do not really rely on $X$ and $Y$ being Gaussian; the proof below shows that it is enough for these variables to have a sufficiently nice joint density.
\begin{proof}
Let $p(u,v)$ denote the joint density of $(X,Y)$ and $B_\epsilon(x)$ denote the ball of radius $\epsilon$ centred at $x$, then
\begin{align*}
    \E\big[H(Y-a\cdot X) \, \big|\, X = x \big]&=\lim_{\epsilon\to 0}\frac{\int_{B_\epsilon(x)}\int_\R H(v-a\cdot u)p(u,v)\;dvdu}{\int_{B_\epsilon(x)}\int_\R p(u,v)\;dvdu}=\frac{\int_\R H(v-a\cdot x)p(x,v)\;dv}{\int_\R p(x,v)\;dv}
\end{align*}
where the final equality follows from applying the dominated convergence theorem to the numerator and denominator. When $x=0$, the above limit is the same for all $a\in\R^n$ (including $a=0$), proving the first item. By the same dominated convergence argument, the above limit is unchanged if we replace $B_\epsilon(x)$ by $\{u\in\R^n\;|\;\lvert Au-Ax\rvert\leq\epsilon\}$. This proves the second item on setting $a=0$ and $H(Y)=Y$.
\end{proof}

The next lemma contains the essence of the divided difference method. It shows that, after appropriate linear transformation (analogous to the mapping $(f(x),f(y))\mapsto (f(x),D_{x,y}(f))$ described in Section~\ref{s:introdd}), the vector $(\nabla f(0),\nabla f(x),\nabla f(y))$ converges to a non-degenerate Gaussian vector as $x ,y \to 0$.

Recall that
\begin{displaymath}
\D =\left\{(x,y)\in\R^{2d}\;\middle|\;0<\lvert x\rvert<\lvert y\rvert<\lvert x-y\rvert\leq 1 \right\},
\end{displaymath}
and for any $g\in C^4(\Lambda_1)$ let $G^3_{0,x}(g)$ denote the vector $(\partial^\alpha g(z)\;|\;\lvert\alpha\rvert\leq 3,z \in \{ 0,x \})$.

\begin{lemma}\label{l:divided_difference}
Let $(x_n,y_n)$ be any sequence in $\D$ converging to $(0,y)$ for some $y$ (possibly equal to zero). There exists a subsequence $n_k$, a sequence of matrices $M_k\in\R^{3d\times 3d}$, and a matrix $M$, such that:
\begin{enumerate}
    \item $\det M_k=\lvert x_{n_k}\rvert^{-d}\lvert y_{n_k}\rvert^{-d-1}(\lvert y_{n_k}\rvert +\sin\theta_{n_k})^{-d+1}$, where $\theta_{n_k}\in[\pi/3,\pi]$ denotes the angle between $x_{n_k}$ and $y_{n_k}$.
    \item For any $g\in C^4(\Lambda_1)$, as $k \to \infty$,
    \begin{displaymath}
    M_k(\nabla g(0),\nabla g(x_{n_k}),\nabla g(y_{n_k}))^t- MG^3_{0,y}(g)=o(\|g\|_{C^4(\Lambda_1)}).
    \end{displaymath}
    \item $MG_{0,y}^3(f)$ is a non-degenerate Gaussian vector.
\end{enumerate}
\end{lemma}

\begin{remark}
    The intuition behind this lemma is that $MG_{0,y}^3(f)$ will be related to one of the four Gaussian vectors in Remark~\ref{r:nondegen} (which are known to be non-degenerate) and the matrices $M_k$ will be chosen so that $M_k(\nabla g(0),\nabla g(x_{n_k}),\nabla g(y_{n_k}))^t$ is a Taylor expansion for $MG_{0,y}^3(g)$. The choice of vector (and matrices) will depend on the sequence $(x_n,y_n)$ and so rather than give several different cases above, we prefer to state the lemma in an abstract way.
\end{remark}
\begin{proof}
By compactness of $[\pi/3,\pi]$, we may pass to a subsequence $\theta_{n_k}$ for which ${\theta_{n_k}\to\theta\in[\pi/3,\pi]}$. For notational clarity we will denote the subsequence by $\theta_n$ (and the matrices in the statement of the lemma will be denoted by $M_n$). Similarly we may assume the following convergence:
\begin{equation}\label{e:Subseq_conv}
\begin{aligned}
\hat{x}_n&:=\frac{x_n}{\lvert x_n\rvert}\to v_1\in \mathbb{S}^{d-1},\quad\hat{y}_n:=\frac{y_n}{\lvert y_n\rvert}\to v_2\in \mathbb{S}^{d-1},\quad\frac{\lvert x_n\rvert}{\lvert y_n\rvert}\to\alpha\in[0,1],\\
\hat{w}_n&:=\left\{\!\begin{aligned}&\frac{\hat{y}_n-\cos(\theta_n)\hat{x}_n}{\sin(\theta_n)}&\text{if }\theta_n\neq \pi\\
&0 &\text{if }\theta_n=\pi
\end{aligned}\right\}\to v_3\in \mathbb{S}^{d-1} \cup\{0\}
,\quad\text{and}\quad \frac{\sin\theta_n}{\lvert y_n\rvert}\to\beta\in[0,\infty].
\end{aligned}
\end{equation}
There are now two different cases to consider: (i) $y=0$, and (ii) $y\neq 0$. 

\textbf{Case (i). $y=0$.} We take $g\in C^4(\Lambda_1)$, and define $M_n$ as the matrix which results in the following row operations:
\begin{align*}
    \begin{pmatrix}
    \nabla g(0)\\
    \nabla g(x_n)\\
    \nabla g(y_n)
    \end{pmatrix}&\to\begin{pmatrix}
    \nabla g(0)\\
    \frac{\nabla g(x_n)-\nabla g(0)}{\lvert x_n\rvert}\\
    \frac{\nabla g(y_n)-\nabla g(0)}{\lvert y_n\rvert}
    \end{pmatrix}\to\begin{pmatrix}
\nabla g(0)\\
\frac{\nabla g(x_n)-\nabla g(0)}{\lvert x_n\rvert}\\
\left(\frac{1}{\lvert y_n\rvert+\sin\theta_n}(I-\hat{x}_n\hat{x}_n^t)+\frac{1}{\lvert y_n\rvert}\hat{x}_n\hat{x}_n^t\right)\frac{\nabla g(y_n)-\nabla g(0)}{\lvert y_n\rvert}
\end{pmatrix}\\
&\to\left(\nabla g(0),\frac{\nabla g(x_n)-\nabla g(0)}{\lvert x_n\rvert},S_n\right)^t
\end{align*}
where
\begin{equation}\label{e:Taylor}
\begin{aligned}
S_n:=&(I-\hat{x}_n\hat{x}_n^t)\frac{1}{\lvert y_n\rvert+\sin\theta_n}\left(\frac{\nabla g(y_n)-\nabla g(0)}{\lvert y_n\rvert}-\cos\theta_n\frac{\nabla g(x_n)-\nabla g(0)}{\lvert x_n\rvert}\right)\\
    &+\frac{1}{\lvert y_n\rvert}\left(\hat{x}_n\cdot\frac{\nabla g(y_n)-\nabla g(0)}{\lvert y_n\rvert}-\hat{y}_n\cdot\frac{\nabla g(x_n)-\nabla g(0)}{\lvert x_n\rvert}\right)\hat{x}_n.
\end{aligned}
\end{equation}
The first statement of the lemma follows from these row operations and noting that 
\begin{displaymath}
\det\left(\frac{1}{\lvert y_n\rvert+\sin\theta_n}(I-\hat{x}_n\hat{x}_n^t)+\frac{1}{\lvert y_n\rvert}\hat{x}_n\hat{x}_n^t\right)=\lvert y_n\rvert^{-1}(\lvert y_n\rvert +\sin\theta_n)^{-d+1}.
\end{displaymath}

By a Taylor expansion, for $\lvert x_n\rvert<1$ we have
\begin{displaymath}
\left|\frac{\nabla g(x_n)-\nabla g(0)}{\lvert x_n\rvert}-\partial_{\hat{x}_n}\nabla g(0)\right|\leq c^\prime\lvert x_n\rvert\|g\|_{C^4(\Lambda_1)}
\end{displaymath}
where $c^\prime$ is an absolute constant. Using the fact that $\hat{x}_n\to v_1$, 
\begin{displaymath}
\left\lvert\frac{\nabla g(x_n)-\nabla g(0)}{\lvert x_n\rvert}- \partial_{v_1}\nabla g(0)\right\rvert\leq c^\prime(\lvert x_n\rvert+\lvert\hat{x}_n-v_1\rvert)\|g\|_{C^4(\Lambda_1)}=o(\|g\|_{C^4(\Lambda_1)}).
\end{displaymath}
Similarly by a Taylor expansion we have
\begin{equation}\label{e:T_n}
\left\lvert\frac{\nabla g(y_n)-\nabla g(0)}{\lvert y_n\rvert}-\cos(\theta_n)\frac{\nabla g(x_n)-\nabla g(0)}{\lvert x_n\rvert}-T_n\right\rvert\leq c^\prime\lvert y_n\rvert^2\|g\|_{C^4(\Lambda_1)}
\end{equation}
where
\begin{equation*}
T_n=\sin(\theta_n)\partial_{\hat{w}_n}\nabla g(0)+\frac{\lvert y_n\rvert}{2}\partial_{\hat{y}_n}^2 \nabla g(0)-\frac{\lvert x_n\rvert}{2}\cos(\theta_n)\partial_{\hat{x}_n}^2\nabla g(0)
\end{equation*}
(note that if $\hat{w}_n=0$ then we take $\partial_{\hat{w}_n}$ to be the identically zero operator). Using the convergence in \eqref{e:Subseq_conv} we see that
\begin{displaymath}
\frac{T_n}{\lvert y_n\rvert+\sin\theta_n}=\frac{\beta}{\beta+1}\partial_{v_3}\nabla g(0)+\frac{1+\alpha}{2(\beta+1)}\partial_{v_1}^2\nabla g(0)+o(\|g\|_{C^4(\Lambda_1)})
\end{displaymath}
as $n\to\infty$. By Taylor expanding the second part of \eqref{e:Taylor} analogously, we see that
\begin{displaymath}
S_n=S+o(\|g\|_{C^4(\Lambda_1)})
\end{displaymath}
where $S$ is defined as
\begin{displaymath}
(I-v_1v_1^t)\left(\frac{\beta}{\beta+1}\partial_{v_3}\nabla g(0)+\frac{1+\alpha}{2(\beta+1)}\partial_{v_1}^2\nabla g(0)\right)+\frac{1}{2}\left(\partial_{v_2}^2\partial_{v_1}g(0)-\alpha\partial_{v_2}\partial_{v_1}^2g(0)\right)v_1.
\end{displaymath}
Hence choosing $M$ such that
\begin{equation}\label{e:DClimit}
MG_{0,y}^3(g)=\left(\nabla g(0),
\partial_{v_1}\nabla g(0),S\right)
\end{equation}
proves the second point of the lemma (for the case $y=0$).

We now let $g=f$ and choose an orthonormal basis $u_1,\dots,u_d$ of $\R^d$ such that $u_1=v_1$ and if $v_3\neq 0$ then $u_2=v_3$. By the first point of Lemma~\ref{l:dc}, the degeneracy of the Gaussian vector on the right hand side of \eqref{e:DClimit} is unchanged if we multiply each of the terms $\nabla g(0)$, $\partial_{v_1}\nabla g(0)$ and $S$ by the orthogonal matrix $A$ with $i$-th row equal to $u_i$. Note that the effect of this multiplication on the first two terms is equivalent to changing the basis for the gradient operator (i.e.\ $A\nabla=(\partial_{u_1},\dots,\partial_{u_d})$). Since $Av_1=(1,0,\dots,0)^t$ we see that when $\beta=\infty$, the right hand side of \eqref{e:DClimit} is non-degenerate if and only if
\begin{displaymath}
\DC\left(A\nabla f(0),\partial_{u_1}A\nabla f(0),
\frac{1}{2}\left(\partial_{u_1}\partial_{v_2}^2f(0)-\alpha \partial_{u_1}^2\partial_{v_2}f(0)\right),
\partial_{u_2}\partial_{u_2} f(0),\dots,
\partial_{u_2}\partial_{u_d}f(0)
\right)
\end{displaymath}
is non-zero, which holds by the fourth point of Remark~\ref{r:nondegen} about the non-degeneracy of various partial derivatives of $f$  (and the first two points of Lemma~\ref{l:dc}).

If $\beta<\infty$ then $\theta_n\to\pi$ and so $v_2=-v_1$. In this case by the first two points of Lemma~\ref{l:dc}, we see that the right hand side of \eqref{e:DClimit} is non-degenerate provided that
\begin{displaymath}
\DC\left(A\nabla f(0),\partial_{u_1}A\nabla f(0),
\partial_{u_1}^3f(0),
\partial_{u_1}^2\partial_{u_2} f(0),\dots,
\partial_{u_1}^2\partial_{u_d}f(0)
\right)
\end{displaymath}
is positive. This holds by the third point of Remark~\ref{r:nondegen} and the second point of Lemma~\ref{l:dc} which completes the proof for $y=0$.

\textbf{Case (ii). $y\neq0$.} Now we let $M_n$ be defined by the following sequence of row operations:
\begin{align*}
    \begin{pmatrix}
    \nabla g(0)\\
    \nabla g(x_n)\\
    \nabla g(y_n)
    \end{pmatrix}&\to\begin{pmatrix}
    \nabla g(0)\\
    \frac{\nabla g(x_n)-\nabla g(0)}{\lvert x_n\rvert}\\
    \nabla g(y_n)
    \end{pmatrix}\to\begin{pmatrix}
\nabla g(0)\\
\frac{\nabla g(x_n)-\nabla g(0)}{\lvert x_n\rvert}\\
\lvert y_n\rvert^{-\frac{d+1}{d}}(\lvert y_n\rvert+\sin\theta_n)^{-\frac{d-1}{d}}\nabla g(y_n)\end{pmatrix}
\end{align*}
and we again see that $M_n$ has the correct determinant, proving the first statement of the lemma. By a Taylor expansion, the above expression differs from
\begin{displaymath}
\left(\nabla g(0),\partial_{v_1}\nabla g(0),
\lvert y\rvert^{-\frac{d+1}{d}}(\lvert y\rvert+\sin\theta)^{-\frac{d-1}{d}}\nabla g(y)\right)^t
\end{displaymath}
by a term which is $o(\|g\|_{C^4(\Lambda_1)})$ proving the second statement of the lemma. When $g=f$, the second point of Remark~\ref{r:nondegen} implies that the above Gaussian vector is non-degenerate, proving the final statement of the lemma.
\end{proof}

\begin{proof}[Proof of Lemma~\ref{l:Third_mom_2}]
Define
\begin{displaymath}
A(x,y)=\lvert x\rvert^{-2d}\lvert y\rvert^{-2d-2}(\lvert y\rvert +\sin\theta)^{-2d+2}\DC(\nabla f(0),\nabla f(x),\nabla f(y)).
\end{displaymath}
Since $A(x,y)$ is continuous and strictly positive on $\D$ (by the first point of Remark~\ref{r:nondegen}), and $\overline{\D}$ is compact, it suffices to show that $\liminf_{n}A(x_n,y_n)>0$ for any sequence $(x_n,y_n)\in \D$ converging to $(x,y)\in\partial \D$.

For the sake of contradiction, suppose $(x_n,y_n) \to  (x,y) \in \partial \D$ is such that $A(x_n,y_n)\to 0$. By Lemma~\ref{l:divided_difference} we may pass to a subsequence $n_k$ and find matrices $M,M_k$ such that for any $g\in C^4(\Lambda_1)$
\begin{displaymath}
M_k(\nabla g(0),\nabla g(x_{n_k}),\nabla g(y_{n_k}))^t\to MG^3_{0,y}(g).
\end{displaymath}
In particular this holds almost surely for $g=f$. Since $f$ is Gaussian, the convergence occurs also in $L^2$. Then by the third point of Lemma~\ref{l:dc} (and the expression for $\det M_k$ in Lemma~\ref{l:divided_difference})
\begin{displaymath}
A(x_{n_k},y_{n_k})=(\det M_k)^2\DC(\nabla f(0),\nabla f(x_{n_k}),\nabla f(y_{n_k}))\to\DC\big( MG^3_{0,y}(f)\big).
\end{displaymath}
By the third point of Lemma~\ref{l:divided_difference}, the latter expression is strictly positive, yielding a contradiction.
\end{proof}

A final ingredient in the proof of Lemma \ref{l:Third_mom_1} is a uniform bound on the conditional moments of $F$:

\begin{lemma}\label{l:C4norm}
For each $k\in\N$ and $\epsilon\in(0,1)$ there exists $c>0$ depending only on the distribution of $f$ such that, for any $(x,y)\in \D$,
\begin{equation}\label{e:Normbound}
\E\big[\|F\|^{k}_{C^{4-\epsilon}(\Lambda_1)}\;\big|\;\nabla F(0)=\nabla F(x)=\nabla F(y)=0\big]<c(1+\|p\|_{C^4(\Lambda_1)})^k.
\end{equation}
\end{lemma}
\begin{proof}
Kolmogorov's theorem \cite[Sections~A.9, A.11.1]{ns16} states that for any centred \mbox{$C^4$-smooth} Gaussian field $h$ there exists a $c>0$, depending on $k$ and $\epsilon$, such that
\begin{displaymath}
\E\big[\|h\|^{k}_{C^{4-\epsilon}(\Lambda_1)}\big]\leq c\Big(\sup_{\lvert\alpha\rvert,\lvert\beta\rvert\leq 4}\sup_{z_1,z_2\in \Lambda_2}\lvert\partial^\alpha_{z_1}\partial^\beta_{z_2}\Cov[h(z_1),h(z_2]\rvert\Big)^{k/2}.
\end{displaymath}
By the Cauchy-Schwarz inequality, the latter expression is bounded above by
\begin{displaymath}
c\Big(\sup_{\lvert\alpha\rvert\leq 4}\sup_{z\in \Lambda_2}\Var[\partial^\alpha h(z)]\Big)^{k/2}.
\end{displaymath}
If we now allow $h$ to be a non-centred $C^4$ Gaussian field, then by the triangle inequality, for each $k\in\N$ and $\epsilon>0$ there exists a $c>0$ such that
\begin{displaymath}
\E\big[\|h\|^{k}_{C^{4-\epsilon}(\Lambda_1)}\big]\leq c\Big(\sup_{\lvert\alpha\rvert\leq 4}\sup_{z\in \Lambda_1}\lvert \E[\partial^\alpha h(z)]\rvert^k +\sup_{\lvert\alpha\rvert\leq 4}\sup_{z\in \Lambda_2}\Var[\partial^\alpha h(z)]^{k/2}\Big).
\end{displaymath}

Denote the conditioning event $A_{x,y}=\{\nabla F(w)=0\;\text{for all }w=0,x,y\}$. By the above inequality, the lemma is proved if we verify the following: there exists a $c>0$, depending only on the distribution of $f$, such that
\begin{equation}\label{e:TwoRequirements}
\begin{aligned}
&\sup_{z\in \Lambda_2}\sup_{(x,y)\in \D}\sup_{\lvert\alpha\rvert\leq 4}\Var[\partial^\alpha F(z)\;|\;A_{x,y}]<c\quad\text{and}\\
&\sup_{z\in \Lambda_1}\sup_{(x,y)\in \D}\sup_{\lvert\alpha\rvert\leq 4}\lvert\E(\partial^\alpha F(z)\;|\;A_{x,y})\rvert<c\|p\|_{C^4(\Lambda_1)}.
\end{aligned}
\end{equation}
Since conditioning can only reduce the variance of a Gaussian variable
\begin{displaymath}
\Var[\partial^\alpha F(z)\;|\;A_{x,y}]\leq\Var[\partial^\alpha F(z)]=\Var[\partial^\alpha f(z)], 
\end{displaymath}
and the latter is bounded uniformly over $z$ and $\alpha$. This verifies the first part of \eqref{e:TwoRequirements}.

Turning to the second part of \eqref{e:TwoRequirements}, for $g\in C^2(\Lambda_1)$, we define
\begin{displaymath}
G_{x,y}(g)=(\nabla g(0),\nabla g(x),\nabla g(y))^t.
\end{displaymath}
By Gaussian regression
\begin{equation}\label{e:GReg}
\begin{aligned}
\E[\partial^\alpha F(z)\;&|\;A_{x,y}]\\
&=\E[\partial^\alpha F(z)]-\Cov[\partial^\alpha F(z),G_{x,y}(F)]\Cov[G_{x,y}(F)]^{-1}\E[G_{x,y}(F)]\\
&=\partial^\alpha p(z)-\Cov[\partial^\alpha f(z),G_{x,y}(f)]\Cov[G_{x,y}(f)]^{-1}G_{x,y}(p).
\end{aligned}
\end{equation}
Note that this verifies the second part of \eqref{e:TwoRequirements} when $p\equiv 0$.

Suppose that the second part of \eqref{e:TwoRequirements} fails, so there exists a sequence $p_n\in C^4(\Lambda_1)$ and $(x_n,y_n,z_n)\in \D\times \Lambda_1$ such that
\begin{displaymath}
\lvert \E[\partial^\alpha F(z_n)\;|\;A_{x_n,y_n}]\rvert\|p_n\|_{C^4(\Lambda_1)}^{-1}\to\infty.
\end{displaymath}
By compactness, we may pass to a subsequence for which $z_{n_k}\to z\in \Lambda_1$ and $(x_{n_k},y_{n_k})\to(x,y)\in\overline{\D}$. Note that if $(x,y)\in\D$ then the second line of \eqref{e:GReg} is finite by the first point of Remark~\ref{r:nondegen}, so evidently $(x,y)\in\partial \mathcal{D}$. Hence we may find matrices $M,M_k$ satisfying the conclusions of Lemma~\ref{l:divided_difference}. Then by Lemma~\ref{l:Gaussian_conditioning} and Gaussian regression
\begin{multline*}
\begin{aligned}
\E[\partial^\alpha F(z_{n_k})\;|\;A_{x_{n_k},y_{n_k}}]
&=\partial^\alpha p_{n_k}(z_{n_k})-\Cov[\partial^\alpha f(z_{n_k}),M_kG_{x_{n_k},y_{n_k}}(f)]\times
\end{aligned}
\\
\Cov[M_kG_{x_{n_k},y_{n_k}}(f)]^{-1}M_kG_{x_{n_k},y_{n_k}}(p_{n_k}).
\end{multline*}
By the second and third parts of Lemma~\ref{l:divided_difference}, we have
\begin{align*}
\Cov[\partial^\alpha f(z_{n_k}),M_kG_{x_{n_k},y_{n_k}}(f)]&\to\Cov[\partial^\alpha f(z),MG^3_{0,y}(f)]\\
\Cov[M_kG_{x_{n_k},y_{n_k}}(f)]^{-1}&\to\Cov[MG^3_{0,y}(f)]^{-1}\\
\lvert M_kG_{x_{n_k},y_{n_k}}(p_{n_k})-MG_{x,y}(p_{n_k})\rvert&=o(\|p_{n_k}\|_{C^4(\Lambda_1)}).
\end{align*}
Since the first two matrices on the right hand side have bounded elements and
\begin{displaymath}
\max\{\lvert\partial^\alpha p_{n_k}(z_{n_k})\rvert,\|G_{x,y}(p_{n_k})\|_\infty\}\leq\|p\|_{C^4(\Lambda_1)} ,
\end{displaymath}
we see that $\lvert \E[\partial^\alpha F(z_{n_k}) | A_{x_{n_k},y_{n_k}}]\rvert \|p_{n_k}\|_{C^4(\Lambda_1)}^{-1}$ is uniformly bounded, yielding the desired contradiction. This completes the proof of \eqref{e:TwoRequirements} and hence of the lemma.
\end{proof}

\begin{proof}[Proof of Lemma~\ref{l:Third_mom_1}]
In the proof $c' > 0$ are constants that depend only on the field and may change from line to line. Recall the event $A_{x,y}=\{\nabla F(w)=0\;\text{for all }w=0,x,y\}$. By H\"older's inequality we have
\begin{displaymath}
 \E \big[ \big| \det(\nabla^2 F(0) \nabla^2 F(x) \nabla^2 F(y) ) \big| \, \big| \, A_{x,y} \big] \leq \prod_{z \in \{0,x,y\}} \E\left[\lvert \det(\nabla^2 F(z))\rvert^3\,\middle|\,A_{x,y}\right]^{1/3} .
\end{displaymath}
Consider $z \in \{0,x,y\}$. For any orthonormal basis $u_1,\dots,u_d$, expanding the determinant we have
\begin{align}
\nonumber \E \big[ \lvert  \det( \nabla^2 F(z) ) \rvert^3 \, \big| \, A_{x,y}\big]^{1/3} &= \E \Bigg[ \Bigg\lvert \sum_{\sigma\in S_d} \prod_{i=1}^d \partial_{u_i} \partial_{u_{\sigma(i)}} F(z)\Bigg\rvert^3 \, \Bigg| \, A_{x,y} \Bigg]^{1/3} \\
\label{e:Moment_determinant}  &\leq c'\sum_{\sigma\in S_d}\prod_{i=1}^d\E \Big[\Big\lvert\partial_{u_i}\partial_{u_{\sigma(i)}}F(z)\Big\rvert^{3d}\,\Big|\,A_{x,y}\Big]^{\frac{1}{3d}},
\end{align}
where $S_d$ is the group of permutations of $\{1,\dots,d\}$ and for the second line we have used H\"older's inequality.

Our aim is to bound the right hand side of \eqref{e:Moment_determinant}. Let us consider the case $z = 0$, in which we claim it is bounded by $c' (1+\|p\|_{C^4(\Lambda_1)})\lvert y\rvert(\sqrt{\lvert y\rvert}+\sin(\theta))$. We consider an orthonormal basis of $\R^d$, $\hat{y},u_2,\dots,u_d$, where $\hat{y}=y/\lvert y\rvert$. Each product on the right hand side of \eqref{e:Moment_determinant} contains either the term corresponding to $\partial_{\hat{y}}^2F(0)$ or two terms of the form $\partial_{u_i}\partial_{\hat{y}}F(0)$, $\partial_{u_j}\partial_{\hat{y}}F(0)$. We will prove the following bounds for these terms:
\begin{equation}\label{e:Moment_bound_1}
\begin{aligned}
\E \big[\big\lvert\partial_{\hat{y}}\partial_{u_i}F(0)\big\rvert^{3d}\,\big| \, A_{x,y}\big]&\leq c'(1+\|p\|_{C^4(\Lambda_1)})^{3d} \lvert y\rvert^{3d} \quad \text{and}\\ 
\E\big[\big\lvert\partial_{\hat{y}}^2F(0)\big\rvert^{3d}\,\big|\,A_{x,y}\big]&\leq c' (1+\|p\|_{C^4(\Lambda_1)})^{3d}\lvert y\rvert^{3d}(\sqrt{\lvert y\rvert}+\sin\theta)^{3d}.
\end{aligned}
\end{equation}
This will be sufficient for the desired bound on \eqref{e:Moment_determinant} since for any $i$ and $\sigma$, the remaining terms satisfy
\begin{displaymath}
\E \big[ \big\lvert\partial_{u_i}\partial_{u_{\sigma(i)}}F(0)\big\rvert^{3d}\, \big| \, A_{x,y} \big] \leq\E\big[\|F\|^{3d}_{C^2(\Lambda_1)}\;\big|\;A_{x,y}\big]\leq c^\prime(1+\|p\|_{C^4(\Lambda_1)})^{3d}
\end{displaymath}
by \eqref{e:Normbound}.

For $z\in\{x,y\}$ let $D(z)=\frac{\nabla F(z)-\nabla F(0)}{\lvert z\rvert}$. Turning to \eqref{e:Moment_bound_1}, by a Taylor expansion we have for any unit vector $v$
\begin{displaymath}
\left\lvert v\cdot D(z)-\partial_{\hat{z}}\partial_{v}F(0)\right\rvert\le c^\prime\lvert z\rvert\;\|F\|_{C^3(\Lambda_1)}.
\end{displaymath}
Then by Lemmas~\ref{l:Gaussian_conditioning} and~\ref{l:C4norm} we have
\begin{equation}\label{e:Moment_bound_2}
\begin{aligned}
\E\big[\big\lvert\partial_{\hat{y}}\partial_{v}F(0)\big\rvert^{3d}\,\big|\,A_{x,y}\big]&\leq\E\Big[\Big\lvert\partial_{\hat{y}}\partial_{v}F(0)-v\cdot D(y)\Big\rvert^{3d}\;\Big|\;A_{x,y}\Big]\\
&\leq c^\prime\lvert y\rvert^{3d}\E\big[\|F\|_{C^3(\Lambda_1)}^{3d}\;\big|\;A_{x,y}\big]\leq c'(1+\|p\|_{C^4(\Lambda_1)})^{3d}\lvert y\rvert^{3d}.
\end{aligned}
\end{equation}
Taking $v=u_i$ establishes the first bound in \eqref{e:Moment_bound_1} while taking $v=\hat{y}$ establishes the second bound in \eqref{e:Moment_bound_1} for $\theta\in[\pi/3,2\pi/3]$.

To prove the second bound for $\theta\in[2\pi/3,\pi]$ we require a higher order Taylor expansion. Specifically by combining the elementary estimates
\begin{align*}
    \left\lvert\hat{y}\cdot D(y)-\partial_{\hat{y}}^2F(0)-\frac{\lvert y\rvert}{2}\partial_{\hat{y}}^3F(0)\right\rvert\leq c^\prime\lvert y\rvert^{3/2}\|F\|_{C^{7/2}(\Lambda_1)}
\end{align*}
and
\begin{align*}
\left\lvert\hat{x}\cdot D(y)-\hat{y}\cdot D(x)-\frac{\cos\theta}{2}(\lvert y\rvert-\lvert x\rvert\cos\theta)\partial_{\hat{y}}^3F(0)\right\rvert\leq c^\prime\lvert y\rvert(\sqrt{\lvert y\rvert}+\sin\theta)\|F\|_{C^{7/2}(\Lambda_1)}
\end{align*}
with Lemma~\ref{l:Gaussian_conditioning} we have
\begin{multline*}
\E\big[\big\lvert\partial_{\hat{y}}^2F(0)\big\rvert^{3d}\,\big|\,A_{x,y}\big]\\
\begin{aligned}
&=\E\left[\Big\lvert\partial_{\hat{y}}^2F(0)-\hat{y}\cdot D(y)+\frac{\lvert y\rvert}{\cos\theta(\lvert y\rvert-\lvert x\rvert\cos\theta)}(\hat{x}\cdot D(y)-\hat{y}\cdot D(x))\Big\rvert^{3d}\;\Big|\;A_{x,y}\right]\\
&\leq c^\prime \lvert y\rvert^{3d}(\sqrt{\lvert y\rvert}+\sin\theta)^{3d}\E\big[\|F\|^{3d}_{C^{7/2}(\Lambda_1)}\:\big|\;A_{x,y}\big]\\
&\leq c^\prime(1+\|p\|_{C^4(\Lambda_1)})^{3d} \lvert y\rvert^{3d}(\sqrt{\lvert y\rvert}+\sin\theta)^{3d}
\end{aligned}
\end{multline*}
completing the proof of \eqref{e:Moment_bound_1}.

To complete the proof of the lemma, we require the bounds
\begin{align*}
\E\big[\lvert\det ( \nabla^2 F(x)) \rvert^{3d}\,\big|\,A_{x,y}\big]^{1/3d}&\leq c'(1+\|p\|_{C^4(\Lambda_1)})^{3d} \lvert x\rvert,\quad\text{and}\\
\E\big[\lvert\det ( \nabla^2 F(y) ) \rvert^{3d}  \,\big|\,A_{x,y}\big]^{1/3d}&\leq c' (1+\|p\|_{C^4(\Lambda_1)})^{3d}\lvert y\rvert.
\end{align*}
These both follow from the arguments given above on swapping the roles of $0$, $x$ and $y$. Specifically the analogue of \eqref{e:Moment_determinant} also holds for both $x$ and $y$ and the analogue of the first part of \eqref{e:Moment_bound_1} is enough to conclude.
\end{proof}

%%%%%%%%%%%%%%%%%%%%%%%%%%%%%%%%%%%%%%%%%%%%%%
%% Single Appendix:                         %%
%%%%%%%%%%%%%%%%%%%%%%%%%%%%%%%%%%%%%%%%%%%%%%
%\begin{appendix}
%\section*{???}%% if no title is needed, leave empty \section*{}.
%\end{appendix}
%%%%%%%%%%%%%%%%%%%%%%%%%%%%%%%%%%%%%%%%%%%%%%
%% Multiple Appendixes:                     %%
%%%%%%%%%%%%%%%%%%%%%%%%%%%%%%%%%%%%%%%%%%%%%%
\begin{appendix}

\section{Basic properties of smooth Gaussian fields}\label{app:Gaussian}

We collect some basic properties of smooth Gaussian fields. First we consider Gaussian fields defined by a stationary moving average representation $f = q \ast W$.

\begin{lemma}
\label{l:wnr}
Let $q \in C^k(\R^d)$, and suppose there exist $c > 0$ and $\beta > d$ such that, for $|x| \ge 1$,
\[  \sup_{|\alpha| \le k} |\partial^\alpha q(x)| \le c |x|^{-\beta}  .   \]
Define $f = q \ast W$. Then $f$ is a.s.\ $C^{k-1}$-smooth, and for every $|\alpha| \le k-1$,
\[ \partial^\alpha f = ( \partial^\alpha q ) \star W . \]
Moreover the spectral measure of $f$ has a continuous density, and in particular its support contains an open set.
\end{lemma}

\begin{proof}
See \cite[Proposition 3.3]{mv20} for a proof of the first statement in the case $d=2$, with the general case identical. The second statement follows from dominated convergence, whilst the last is a consequence of the Riemann-Lebesgue lemma.
\end{proof}
 
We next state a standard non-degeneracy property of stationary Gaussian fields:
 
\begin{lemma}
\label{l:nondegen}
Let $f$ be a $C^k$-smooth stationary Gaussian field on $\R^d$, and suppose that the support of it spectral measure $\mu$ contains an open set. Consider a finite collection $(x_i, \alpha_i)_{1 \le i \le n}$, where $x_i \in \R^d$, $\alpha_i$ is a multi-index such that $|\alpha_i| \le k$, and each $(x_i,\alpha_i)$ is distinct. Then the Gaussian vector
\[ \big( \partial^{\alpha_i} f(x_i) \big)_{1 \le i \le n}   \]
is non-degenerate. 
\end{lemma}

\begin{proof}
 See \cite[Theorem 6.8]{wen05} for the case $k=0$, and \cite[Lemma A.2]{bmm20a} for the case $k \le 2$; the general case can be proven identically.
\end{proof}

\medskip
\section{A topological lemma}
\label{s:top}

In this section we deduce Lemma~\ref{l:Topological_stability} from a fundamental lemma of stratified Morse theory, which states roughly that the topology of the level set $\{g + t p = 0\}$ does not change as $t$ varies over $[t_1,t_2]$ unless, for some $t \in [t_1,t_2]$, $g + t p$ has a (stratified) critical point at level $0$.

Although this lemma is classical, there are complicating factors in our setting: most of the literature treats either constant perturbations or compact manifolds without boundary, whereas we consider general perturbations on domains with corners. However our setting is simpler in one respect: the functions $g$ and $p$ are defined in a \textit{neighbourhood} of the~domains.

Let $D$ be a stratified domain as defined in Section \ref{s:stability} and let $U$ be a compact domain that contains a neighbourhood of $D$. For two subsets $A,B \subset \mathrm{Int}D$, an \textit{isotopy} between $A$ and $B$ is a continuous map $H: A \times [t_1,t_2] \to \mathrm{Int}D$ such that, for each $t \in [t_1,t_2]$, $H(\cdot,t)$ is a  homeomorphism, $H(\cdot,t_1) = \textit{id}$, and $H(A,t_2) = B$.

Fix a pair of functions $(g,p) \in C^2(U) \times C^2(U)$, and for $t \in [t_1, t_2]$ define $g_t=g+t p$. A \textit{quasi-critical point} is a pair $(x,t) \in D \times [t_1,t_2]$ such that $g_t$ has a (stratified) critical point at $x$ at level $0$. The number of such quasi-critical points is denoted by $N_{q.c.}(D,g,p)$.

Lemma \ref{l:Topological_stability} is an immediate corollary of the following rather standard lemma:

\begin{lemma}[Fundamental lemma of stratified Morse theory]
\label{l: morse continuity}
If ${N_{q.c.}(D,g,p) = 0}$, then there is an isotopy between the union of all interior components (i.e.\ components that do not intersect $\partial D$) of the level sets $\{g_{t_1}=0\}\cap D$ and $\{g_{t_2} = 0\} \cap D$. In particular, the number of interior components of the sets $\{g_{t_1} = 0\}$ and $\{g_{t_2} = 0\}$ are the same. The conclusion is also true for the interior components of the  excursion sets $\{g_{t_i} \ge 0\}\cap D$, $i = 1,2$.
\end{lemma}

Before proving Lemma \ref{l: morse continuity}, let us deduce Lemma \ref{l:Topological_stability}:

\begin{proof}[Proof of Lemma \ref{l:Topological_stability}]
Consider the interpolation $g_t = g - \ell + t p$ for $t \in [0, 1]$. Since $(g,p)$ is stable (on $D$ at level $\ell$), for any $(x,t)\in D\times[0,1]$ we have
\begin{displaymath}
\max\big\{\lvert g_t(x) \rvert,\lvert \nabla_F g_t(x)\rvert\big\}>0.
\end{displaymath}
Hence $N_{q.c.}(D,g-\ell,p) = 0$, and the result follows by an application of Lemma \ref{l: morse continuity}.
\end{proof}

\begin{proof}[Proof of Lemma \ref{l: morse continuity}]
First we prove a local version of the result: we fix $t' \in [t_1,t_2]$ and prove that the conclusion holds for $t \in [t', t'+\delta]$ and some $\delta > 0$.  Without loss of generality we may assume that $t' = 0$.

Let $L_{t}$  be the union of all interior connected components of $ \{g_t = 0\} \cap D$, and let $L_{t,\epsilon}$ be the $\epsilon$-neighbourhood of $L_t$. If we consider $g_t$ as a function $D\times [t_1,t_2]\rightarrow \R$, then its zero locus $Z$ is a manifold. This follows from the implicit function since the derivative is non-vanishing on the locus. Indeed, let us assume that it contains a critical point, then at this point $(x,t)$ we have $\partial_t g_t(x)=0$ and $\nabla_x g_t(x)=0$, hence there is a quasi-critical point.  This proves that $Z$ is a differentiable manifold of co-dimension $1$ and its boundary lies on the boundary of $D\times [t_1,t_2]$. Let us consider an interior component of $L_0$. This component is  on the boundary of a certain component of $Z$. We claim that this component does not intersect $(\partial D)\times [t_1,t_2]$. If it does, then we can consider the first time $\tau$ when it happens, and then the intersection of the component by $D\times\{\tau\}$ contains a level surface of $g_\tau$ tangent to the boundary. This means that $g_\tau$ has a stratified critical point, which is impossible by our assumptions. So each component of $Z$ either lies in $\mathrm{Int}D\times[t_1,t_2]$ or its every section intersects $\partial D$. 

Since this compact surface does not intersect $(\partial D)\times [t_1,t_2]$, the distance between them is strictly positive. In other words, there is $\epsilon>0$ such that $L_{t,\epsilon}\subset \mathrm{Int}D$ for all $t\in [t_1,t_2]$. Similarly, by reducing $\epsilon$ if necessary, there is $c>0$ such that $|\nabla g_t (x)|>c$ for all ${x\in L_{0,\epsilon}}$ and all $t\in [t_1,t_2]$.  Hence, by choosing $\delta>0$ sufficiently small, we can ensure that ${|\nabla g_t (x)|>(4\delta/\epsilon)\cdot\sup_D\lvert p\rvert}$ for all $t\in [0, \delta]$ and all $x\in L_{0,\epsilon}$. 

For $x_0\in D$, define the flow
\[ 
\frac{dx_t}{dt}=-\partial_t g_t(x_t)\frac{\nabla g_t(x_t)}{|\nabla g_t(x_t)|^2} \ , \quad t \in [0,\delta] ,
\]
which is well-defined outside critical points of $g_t$. The key property of this flow is that, by the chain rule, $g_t(x_t)$ is constant in $t$. Note also that
\[  
\Big| \frac{dx_t}{dt} \Big| = \lvert p(x_t)\rvert |\nabla g_t(x_t)|^{-1} . 
\]
Let us consider $x_0\in L_0$. By the assumption that $|\nabla g_t(x)| >  (4 \delta / \eps)\cdot\sup_D\lvert p\rvert$ for all $x \in L_{0,\epsilon}$, we deduce that $|x_t-x_0| \le \eps/4$. Hence $H(x,t)=x_t$ is an isotopy such that ${H(L_0,t)\subset L_t\cap L_{0,\epsilon}}$. Note that in terms of the set $Z$ described above, for each component of $L_0$ its images under the flow are the sections of the corresponding component of $Z$. Our flow argument defines an isotopy between sections of all `interior' components of $Z$. In particular, this means that the union of all `interior' components of $Z$ is topologically $L_0\times[t_1,t_t]$. Hence $H(L_0,t)=L_t$ for all $t\in [0,\delta]$. 

This argument shows that local isotopies exist: for each $t$ there is $\delta=\delta(t)$ such that $L_t$ is isotopic to $L_{t+\delta(t)}$. By compactness we can choose a finite number of times $s_k$ such that
\[ s_1=t_1<s_2<s_1+\delta(s_1)<s_3<s_2+\delta(s_2)<\dots<t_2<s_n+\delta(s_n).  \]
By concatenating isotopies on overlapping time intervals we obtain a global isotopy, that is, an isotopy between $L_{t_1}$ and $L_{t_2}$. This completes the proof of the lemma for level sets. 

Since the interior level sets form the boundary of all interior excursion sets, by the isotopy extension lemma \cite[Theorem 8.1.3]{hir76} we can extend the isotopy of $L_t$ to an isotopy of $\mathrm{Int}D$ which agrees with the isotopy of $L_t$ and is the identity in a neighbourhood of $\partial D$. In particular, this gives an isotopy of the union of all interior components of the corresponding excursion sets.
\end{proof}

\begin{remark}
With more work a similar flow could also be defined for components that intersect the boundary, which would define a \emph{stratified} isotopy of $\{g_t=0\}\cap D$ for all $t$. Since we only apply Lemma \ref{l: morse continuity} to deduce the equality of the number of interior components, we do not need this.
\end{remark}

\end{appendix}

\medskip

\bigskip
\bibliographystyle{plain}
\bibliography{paper}

%%%%%%%%%%%%%%%%%%

\end{document}